\documentclass[11pt]{amsart}
\usepackage[utf8]{inputenc}
\usepackage{mathrsfs}
\usepackage{mathtools}
\usepackage{amstext}
\usepackage{amsthm}
\usepackage{amssymb}
\usepackage{graphicx}
\usepackage{tikz}
\usepackage{subfigure}

\usepackage[unicode=true,pdfusetitle,
 bookmarks=true,bookmarksnumbered=false,bookmarksopen=false,
 breaklinks=false,pdfborder={0 0 1},backref=false,colorlinks=false]
 {hyperref}

\makeatletter

\usepackage{enumitem}
\usepackage{algorithm}

\numberwithin{equation}{section}
\numberwithin{figure}{section}
\theoremstyle{plain}
\newtheorem{thm}{\protect\theoremname}[section]
\theoremstyle{remark}
\newtheorem{rem}[thm]{\protect\remarkname}
\theoremstyle{definition}
\newtheorem{definition}[thm]{\protect\definitionname}
\theoremstyle{plain}
\newtheorem{lem}[thm]{\protect\lemmaname}
\theoremstyle{plain}

\newtheorem{cor}{Corollary}[thm]

\makeatother

\providecommand{\definitionname}{Definition}
\providecommand{\factname}{Fact}
\providecommand{\lemmaname}{Lemma}
\providecommand{\remarkname}{Remark}
\providecommand{\theoremname}{Theorem}

\newcommand\bE{\mathbb{E}}
\newcommand\bP{\mathbb{P}}

\title[Curvature-driven manifold fitting]{Curvature-driven manifold fitting under unbounded isotropic noise}
\author{Ruowei Li and Zhigang Yao}
\address{Zhigang Yao: Department of Statistics and Data Science, National University of Singapore, 21 Lower Kent Ridge Road, Singapore 117546;
Center of Mathematical Sciences and Applications, Harvard University, 20 Garden Street, Cambridge USA 02138}
\email{zhigang.yao@nus.edu.sg; zhigang.yao@cmsa.fas.harvard.edu}
\address{Ruowei Li: Department of Statistics and Data Science, National University of Singapore, 21 Lower Kent Ridge Road, Singapore 117546}
\email{ruoweili@nus.edu.sg}

\usetikzlibrary{calc}
\begin{document}

\maketitle
\begin{abstract}
Manifold fitting aims to reconstruct a 
low-dimensional manifold from high-dimensional data, whose framework is established by Fefferman et al. \cite{fefferman2020reconstruction,fefferman2021reconstruction}. 
This paper studies the recovery of a compact \(C^3\) 
submanifold \(\mathcal{M} \subset \mathbb{R}^D\) with dimension $d<D$ and positive reach \(\tau\) from observations \(Y = X + \xi\), where \(X\) is uniformly distributed on \(\mathcal{M}\) and \(\xi \sim \mathcal{N}(0, \sigma^2 I_D)\) denotes isotropic Gaussian noise.  
To project any points $z$ in a tubular neighborhood \(\Gamma\) of \(\mathcal{M}\) 
onto \(\mathcal{M}\),
we construct a sample-based 
estimator $F:\Gamma\to\mathbb{R}^D$ by a normalized local kernel with the theoretically derived bandwidth \(r = c_D\sigma\). 
Under a sample size of \(O(\sigma^{-3d-5})\), we establish with high probability the uniform asymptotic expansion  
\[
F(z) = \pi(z) + \frac{d}{2} H_{\pi(z)} \sigma^2 + O(\sigma^3), \qquad z \in \Gamma,
\]
where \(\pi(z)\) is the projection of \(z\) onto \(\mathcal{M}\) and \(H_{\pi(z)}\) is the mean curvature vector of \(\mathcal{M}\) at \(\pi(z)\). 
The resulting manifold \(F(\Gamma)\) has reach bounded below by \(c \tau\) for \(c>0\) and achieves a state-of-the-art Hausdorff distance of \(O(\sigma^2)\) to \(\mathcal{M}\).
Numerical experiments confirm the quadratic decay of the reconstruction error and demonstrate the computational efficiency of the estimator $F$. Our work provides a curvature-driven framework for denoising and reconstructing manifolds with second-order accuracy.

\end{abstract}

\section{Introduction}



%
Real-world data often exhibit high-dimensional observations with few intrinsic degrees of freedom. This aligns with the manifold hypothesis, which states that such data in $\mathbb{R}^D$ lie near a manifold $\mathcal{M}$ with dimension $d<D$. 
To recover this latent structure, manifold learning methods construct low-dimensional embeddings where Euclidean distances approximate geodesic ones, as seen in
 Isometric Mapping \cite{tenenbaum2000global}, Locally Linear Embedding \cite{roweis2000nonlinear, donoho-hessian-lle}, Laplacian Eigenmaps \cite{belkin2003laplacian}. 
In practice, data are contaminated by noise from measurement or transmission.
Manifold fitting aims to reconstruct a smooth \(\mathcal{M}\) from noisy samples while preserving its intrinsic geometry and topology. 
Applications include denoising single-cell data \cite{yao2024single}, refining RNA structural models \cite{wu2025principal}, recovering data in Cryo-electron microscopy and nanophotonic light detection \cite{kim2021nanophotonics}, and wind-direction analysis \cite{dang2015wind}.
It also provides a foundation for manifold-aware generative models, such as GANs \cite{goodfellow2014generative} and manifold fitting with CycleGAN~\cite{yao2023manifold}.

\subsection{Geometric intuition}
In Riemannian geometry, the tangent space provides a first-order linear approximation to a smooth submanifold. A precise description of its local embedding requires second-order geometric information which is encoded in the mean curvature vector (Figure \ref{Fig:meancurvature}). Defined as the normalized trace of the second fundamental form,  or equivalently the average of the principal curvatures $\kappa_i$ (see Section  \ref{se:Geometric_quantities}). The mean curvature vector is the fundamental second-order extrinsic quantity characterizing the submanifold’s bending toward the normal space.
For instance, on a sphere it points inward, while on 
a saddle it vanishes.

\begin{figure}[ht]
    \centering   
    \includegraphics[width = 0.55\textwidth]{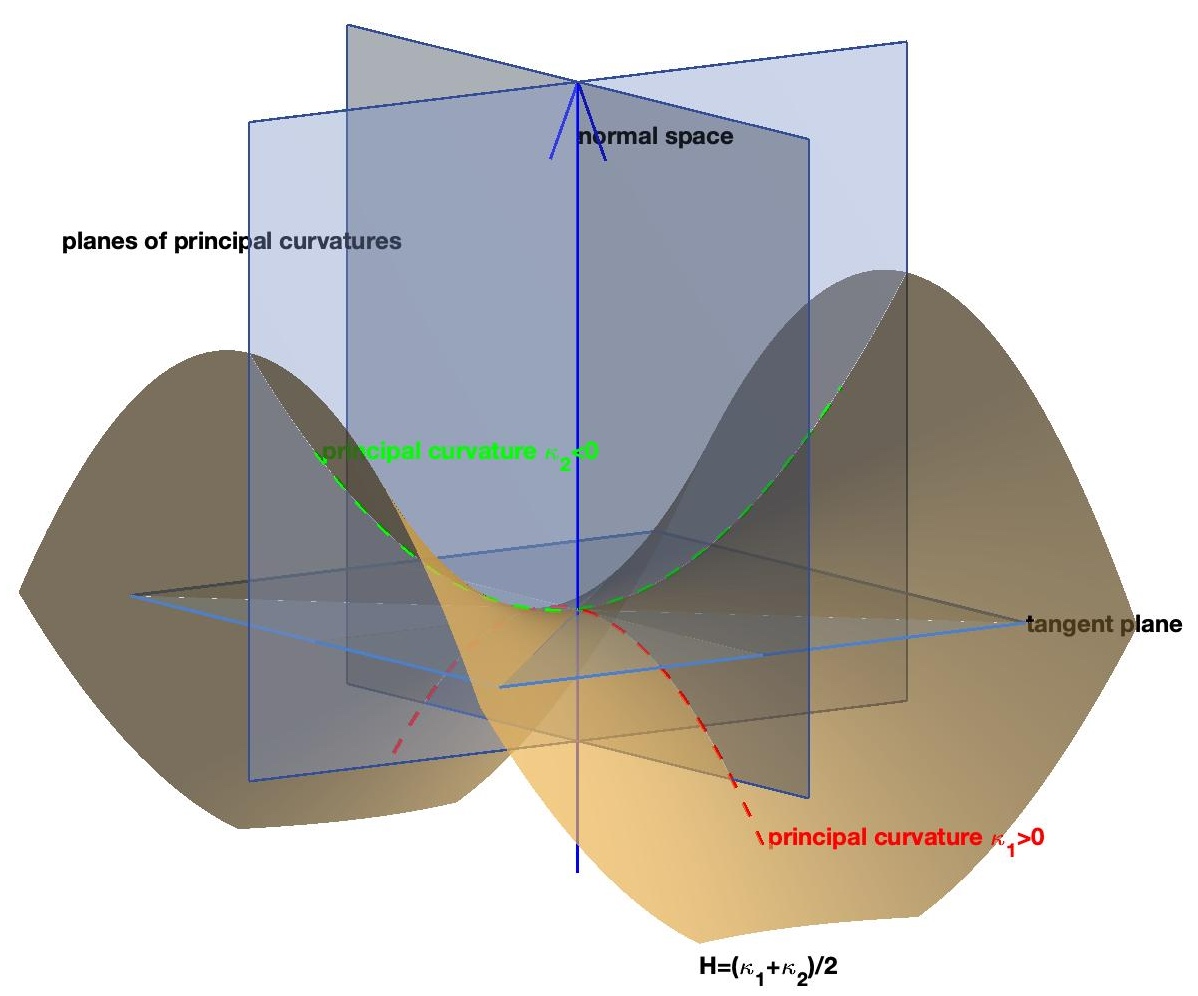}
    \caption{The mean curvature $H$ of a saddle. 
    }\label{Fig:meancurvature}
\end{figure}

Consider the additive Gaussian noise model 
\(Y = X + \xi\), where the random vector \(X\) is uniformly distributed on the \(d\)-dimensional manifold \(\mathcal{M} \subset \mathbb{R}^D,D>d\) and \(\xi \sim \mathcal{N}(0, \sigma^2 I_D)\) represents isotropic Gaussian noise. 
As shown in Figure \ref{fig:noise-samples},
when samples $x$ from the manifold $\mathcal{M}$ are perturbed by isotropic ambient noise (blue dashed circles), the density of noise samples  (simply seen as covered frequency) encodes geometric information of the underlying manifold. Intuitively, 
the mean curvature vector governs the distribution of the noise, resulting in a higher density in the inward normal direction compared to the outward direction (red dashed lines). This phenomenon is also discussed in the framework of ridge theory \cite{genovese2014nonparametric}. The symmetry of the density function within the tangent space along the principal curvature directions leads to a negligible bias. Therefore, the curvature influences the density of the noisy point \(y\), that is, \(p_\sigma(y)\). A precise mathematical characterization is provided in Theorem \ref{th:probability noisy} and Theorem \ref{th:score_noisy}, which distinguishes our work from \cite{genovese2014nonparametric}. This geometric insight guides the construction of our curvature-driven manifold fitting method.


\begin{figure}[htbp]
    \begin{tikzpicture}    
        \draw (0.2 ,0) node[above] {$\mathcal{M}$} arc (50:130:6); 
        
        \draw [dotted] (-8,1.4)--(-4,1.4) node[right] {$x$}--(0,1.4)  node[above] {$T_{x}\mathcal{M}$};
        \draw [blue, densely dotted] (-4,1.4) circle (0.8);
        \draw [blue, densely dotted] (-4.3,1.4) circle (0.8);
        \draw [blue, densely dotted] (-4.6,1.35) circle (0.8);
        \draw [blue, densely dotted] (-4.9,1.3) circle (0.8);
        \draw [blue, densely dotted] (-5.3,1.25) circle (0.8);
        
        \draw [blue, densely dotted] (-3.7,1.4) circle (0.8);
        \draw [blue, densely dotted] (-3.4,1.35) circle (0.8);
        \draw [blue, densely dotted] (-3.1,1.3) circle (0.8);
        \draw [blue, densely dotted] (-2.7,1.25) circle (0.8);

        \draw [blue, densely dotted] (-2.7,1.3)--(-1.95,1.52) node[above = 1pt, left = 2pt,black] {$\sigma$};
        
        \fill  (-4,1.4) circle (1pt);

        \draw [<->, densely dotted] (-7.5,0) -- (-7.5,1.4) node[below = 20pt, left = 0.1pt]{$O(r^2)$};
        \draw [<->, densely dotted] (-4,1.4) -- (-7.5,1.4) node[above=2pt, right = 30pt]{$r$};

        \draw [-, densely dotted] (-4,0) -- (-4,3) node[below = 5pt, right = 3pt]{$T_x^\bot \mathcal{M}$};

        \draw [->, densely dotted] (-4,1.4) -- (-4,0.2) node[below = 1pt, right = 3pt]{$H_x$};

        \draw [<->, red, densely dotted] (-4,0.6) -- (-4,2.2); 

        \draw [<->, red, densely dotted] (-3.75,0.6) -- (-3.6,2.19);
        \draw [<->, red, densely dotted] (-3.45,0.58) -- (-3.2,2.15);
        \draw [<->, red, densely dotted] (-3.2,0.53) -- (-2.85,2.07);
        \draw [<->, red, densely dotted] (-2.93,0.5) -- (-2.48,2);

        \draw [<->, red, densely dotted] (-4.25,0.6) -- (-4.4,2.19);
        \draw [<->, red, densely dotted] (-4.49,0.58) -- (-4.8,2.15);
        \draw [<->, red, densely dotted] (-4.75,0.55) -- (-5.15,2.07);
        \draw [<->, red, densely dotted] (-5.01,0.50) -- (-5.55,1.99);

    \end{tikzpicture}
    \caption{ 
    The figure illustrates samples corrupted by isotropic noise. Points \(x \in \mathcal{M}\) on the underlying submanifold are displaced by isotropic noise \(\sigma\) (blue dashed circles), and the dominant bias of the density of noise appears in the normal direction (red dashed lines). At \(x\), 
     \(T_x \mathcal{M}\), $T_x^\bot \mathcal{M}$ and $H_x$ denote the tangent space, normal space  and mean curvature vector respectively.
     At scale \(r\), \(\mathcal{M}\) deviates from \(T_x \mathcal{M}\) in the normal direction by order  \(O(r^2)\).
}\label{fig:noise-samples}
\end{figure}

\subsection{Current manifold fitting methods} 
We introduce two fundamental techniques in current manifold fitting methods \cite{genovese2014nonparametric,mohammed2017manifold,fefferman2018fitting,yao2019manifold,fefferman2021fitting}. The first technique  (Figure \ref{Fig:PreviousWork}(a)) applies local Principal Component Analysis (PCA) \cite{Tyagi2013Tangent} to samples within a Euclidean ball of radius \(r\) centered on an observed sample \(y\). It identifies the \(d\)-dimensional subspace by minimizing the sum of squared distances to all samples lying in this neighborhood, yielding an estimated tangent space \(\widehat{T_{y}\mathcal{M}}\) as an approximation to the true tangent space \(T_{\pi(y)}\mathcal{M}\) at the projection \(\pi(y)\).
This approach introduces two intrinsic errors:
\[d(\mathcal{M},\widehat{T_{y}\mathcal{M}})=d(\mathcal{M},T_{\pi(y)}\mathcal{M})+d(T_{\pi(y)}\mathcal{M},\widehat{T_{y}\mathcal{M}}).\]
First, the local linear approximation of the manifold by a tangent space incurs a curvature-induced normal deviation of order \(O(r^2)\), as depicted in Figure \ref{fig:noise-samples}. Hence, the principled guidance for selecting the critical scale $r$ or the number $k$ of \(k\)-Nearest Neighbors (\(k\)-NN) is key. Second, the estimated tangent space \(\widehat{T_{y}\mathcal{M}}\) at the noisy point \(y\) inherently differs from the tangent space \(T_{\pi(y)}\mathcal{M}\) at the latent projection point $\pi(y)$. These biases propagate into  subsequent steps such as the estimation of normal spaces~\cite{yao2019manifold} and density functions \cite{yao2023manifold}.
 The second technique (Figure \ref{Fig:PreviousWork}(b)) relies on local averaging of samples within spherical or cylindrical neighborhoods to reduce sampling variability. It is commonly implemented via mean shift \cite{Fukunaga1975} or kernel density estimation (KDE) \cite{Parzen1962}.
The choice of bandwidth $r_i$ is equally critical: a scale too small fails to suppress variance, while one too large introduces geometric bias. A theoretically grounded and precise method for bandwidth selection remains a challenge. 

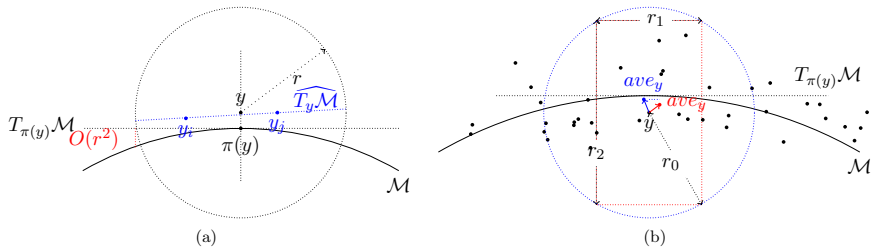
\begin{figure}[htbp]
    \centering
    \resizebox{.95\textwidth}{!}{
        \subfigure[]{
            \begin{tikzpicture}
                \draw (3,-0.8040) node[below] {$\mathcal M$} arc (60:120:6);
                \draw [densely dotted] (3,0) --(-3,0) node[above=0pt, left=0pt] {$T_{\pi (y)}\mathcal M$};
                \draw [red, densely dotted] (-2,0) --(-2,-0.38) node[above=5pt, left=5pt] {$O(r^2)$};
                \draw [densely dotted] (0,0.3) circle (2);
                \fill (0,0) node[below]{$\pi (y)$} circle (1pt);
                \fill (0,0.3) node[above]{$y$} circle (1pt);
                \draw [densely dotted] (0,-1) -- (0,1.5);
                \draw [dotted,->] (0,0.3) -- (1.599,1.5) node[left = 15pt,below=10pt]{$r$};
                
                \fill [blue] (-1.041889, 0.19115348) node[below]{$y_i$} circle (1pt);
                \draw [blue,densely dotted] (-2, 0.1482578) -- (2,0.3688866)
                node[above=5pt, left=0pt] {$\widehat {T_{y}\mathcal M}$};;

                \fill [blue] (0.695935336485, 0.30) node[below]{$y_j$} circle (1pt);
            \end{tikzpicture}
            \label{Fig:locPCA}
        }

        \subfigure[]{
            \begin{tikzpicture}    
                \draw (0 ,0) node[below] {$\mathcal M$} arc (60:120:8);
                \draw [densely dotted] (-6.9,1.0718) --(-4,1.0718) --(-0.6,1.0718) node[above] {$T_{\pi (y)}\mathcal M$};
                \draw[red, densely dotted] (-5,-1) rectangle (-3.,2.5);
                \draw[densely dotted, <->] (-5,2.5) --(-3,2.5) node[left = 16pt,fill=white] {$r_1$};
                \draw[densely dotted, <->] (-5,2.5) --(-5,-1) node[above = 21pt, fill=white] {$r_2$};
                \draw[densely dotted, <->] (-4,0.75) --(-3,-1) node[above = 22pt, left=8pt, fill=white] {$r_0$};
                \fill (-2.9811,0.6126) circle (1pt);
                \fill (0.1380,0.7562) circle (1pt);
                \fill (-5.1115,0.0816) circle (1pt);
                \fill (-5.0900,0.5550) circle (1pt);
                \fill (-1.4464,0.1926) circle (1pt);
                \fill (-1.7007,1.3119) circle (1pt);
               \fill (-5.9252,0.2342) circle (1pt);
              \fill (-4.9937,0.3653) circle (1pt);\fill (-0.1001,0.2172) circle (1pt);\fill (-0.7597,0.9058) circle (1pt);\fill (-3.8094,2.1171) circle (1pt);\fill (-6.0577,0.9415) circle (1pt);\fill (-4.1842,0.5993) circle (1pt);\fill (-1.7753,0.7675) circle (1pt);\fill (-0.9972,0.9182) circle (1pt);\fill (-6.5370,1.6932) circle (1pt);\fill (-3.7693,1.4801) circle (1pt);\fill (-5.3435,0.6265) circle (1pt);\fill (-3.4319,0.7174) circle (1pt);\fill (-5.9580,0.5476) circle (1pt);\fill (-2.8541,1.2264) circle (1pt);\fill (-0.2795,0.4803) circle (1pt);\fill (-0.6471,0.6389) circle (1pt);\fill (-3.2643,0.6250) circle (1pt);\fill (-2.4899,0.6930) circle (1pt);\fill (-2.0460,1.2104) circle (1pt);\fill (-3.3017,2.2170) circle (1pt);\fill (-2.5288,0.4695) circle (1pt);\fill (-5.1595,0.9918) circle (1pt);\fill (-4.5391,1.6005) circle (1pt);\fill (-7.3842,0.3415) circle (1pt);\fill (0.0170,0.4563) circle (1pt);\fill (-7.3544,0.5188) circle (1pt);\fill (-5.9907,0.3842) circle (1pt);\fill (-5.4281,0.5091) circle (1pt);\fill (-6.2409,1.2324) circle (1pt);\fill (-3.7367,1.5424) circle (1pt);\fill (-5.4203,0.5675) circle (1pt);
                
                \fill (-4,0.75) node[below]{$y$} circle (1pt);
                \draw [blue, densely dotted] (-4,0.75) circle (2);
                \fill [blue] (-4.1,1) node[above]{$ave_y$} circle (1pt);
                \fill [red] (-3.8,0.9) node[below,right]{$ave_y$} circle (1pt);
                \draw [blue,densely dotted] (-4,1) -- (-3.8,1);
                \draw [blue,->] (-4,0.75) -- (-4.1,1);    
                \draw [red,->] (-4,0.75) -- (-3.8,0.9); 
            \end{tikzpicture}
            \label{Fig:localmean}
        }
    }
    \caption{(a) Local PCA: For an observed point \(y\), its projection \(\pi(y)\) on \(\mathcal{M}\), the tangent space \(T_{\pi(y)}\mathcal{M}\), and the estimated tangent space \(\widehat{{T}_{y}\mathcal{M}}\) constructed from neighboring points \(y_i, y_j\); (b) Local averaging: The average points \( ave_y\) of observed samples (black dots) within a spherical (blue) or cylindrical (red) neighborhood centered at \(y\).}
    \label{Fig:PreviousWork}
\end{figure}
Then we present the existing manifold fitting methods in detail. Methods based on Delaunay triangulation assume noiseless samples \cite{cheng2005manifold,boissonnat2009manifold}, while deconvolution-based estimators achieve minimax rates under restrictive noise models \cite{genovese2012minimax,genovese2012manifold}. Later, ridge 
estimation via KDE was proposed in \cite{genovese2014nonparametric}. 

The method \cite{mohammed2017manifold} estimates the manifold via an approximate squared-distance function derived from noiseless samples using KDE and local PCA. 
The follow-up work \cite{fefferman2018fitting} introduces a method robust to ambient noise $\sigma$. It constructs an approximate distance function by successively estimating the local normal space from noisy samples via local PCA. The manifold estimator is defined as its zero set, achieving a Hausdorff error of \(O(\sigma)\) and a lower-bounded of reach (see Definition  \ref{def:reach}) with high probability.

The approach in \cite{yao2019manifold} employs the local sample mean to reference the underlying manifold and estimates the normal space via a locally weighted average of normal spaces on samples. Although its theoretical error remains \(O(\sigma)\), this technique substantially reduces computational complexity. 

In \cite{fefferman2021fitting}, point refinement reduces sampling bias by averaging within hyper-cylinders around initial estimates. Integrating this into the framework of \cite{fefferman2018fitting} produces an estimator with Hausdorff error \(O(\sigma^2)\) and a lower-bounded reach. However, the refinement assumes access to noiseless samples from the latent manifold, contradicting the noisy observation model.

The work \cite{yao2023manifold} proposes a two-stage estimator with a sample complexity \(O(\sigma^{-(d+3)})\). A spherical local mean determines projection directions, which then define cylindrical domains where a second local mean to obtain the manifold estimate. This estimator achieves a Hausdorff distance of \(O(\sigma^2\log(1/\sigma))\) and a lower-bounded reach with high probability, though a uniform estimation remains unestablished. 

In this paper, we demonstrate that a local radial mean shift with the theoretically derived bandwidth \(r = c_D \sigma\) effectively projects noisy samples toward the underlying manifold under the additive Gaussian noise model (Figure \ref{fig:estimatorF}). The estimator achieves a state-of-the-art error of \(\frac{d}{2} H_{\pi(z)} \sigma^2\), which is the second-order bias induced by the mean curvature. Thus, our framework performs not merely denoising but a statistical inversion of the data-generating geometry.

\subsection{Main results}

We assume $\mathcal{M}$ is a compact, $C^3$-smooth $d$-dimensional submanifold of $\mathbb{R}^D$, and ambient space $\mathbb{R}^D$ is equipped with the Euclidean metric $g_E$ and $L^2$-norm  $\|\cdot\|$. Suppose $\mathcal{M}$ has finite volume $\operatorname{Vol}(\mathcal{M})=\int_\mathcal{M}d\mu$ and positive reach $\tau>0$. Consider a random vector $Y \in \mathbb{R}^D$ that can be expressed as
\begin{equation}
    \label{eq:Add_model}
    Y = X + \xi,
\end{equation}
where $X \in \mathbb{R}^D$ is an unobserved random vector following the uniform distribution $\omega$ with respect to the induced volume measure $d\mu$ on $\mathcal{M}$, and $\xi \sim \phi_\sigma$ represents the $D$-dimensional isotropic Gaussian noise in ambient-space, 
that is,
$$\phi_\sigma (\xi)= (\frac{1}{2\pi \sigma^2})^{\frac{D}{2}}e^{-\frac{\|\xi\|^2}{2\sigma^2}}.$$ 
The density of $Y$ can be viewed as the convolution of $\omega$ and $\phi_\sigma$, 
\begin{equation}
    \label{eq:def:nu}
    p_\sigma (y) = \int_\mathcal{M} \phi_\sigma(y-x)\omega(x)\,d\mu(x).
\end{equation}



 Assume observed samples $\{y_i\}_{i=1}^N \subset \mathbb{R}^D$ being $N$ independent and identical realizations of $Y$.
Let $\varphi_r:\mathbb{R}^D\to(0,\infty)$ be a smooth compact supported rotation-invariant kernel with bandwidth $r>0$, for example, a truncated Gaussian kernel
\begin{equation}\label{eq:truncated Gaussian kernel}
    \varphi_r(z) =\phi_r (z) \chi\!\big(\tfrac{\|z\|}{\sqrt{2}r}\big)=({2\pi r^2})^{-\frac{D}{2}}e^{\frac{-\|z\|^2}{2r^2}}\,\chi\!\big(\tfrac{\|z\|}{\sqrt{2}r}\big),
\end{equation}
where \(\chi:[0,\infty)\to[0,1]\) is a smooth radial function
satisfying
\[
\chi(t)=1\ \text{ for } 0\le t\le \rho_0<1,\qquad
\chi(t)=0\ \text{ for } t\ge 1.
\]
Then $\varphi_r\in C_0^{\infty}(B_D(0,\sqrt{2}r))$. Let $$r=c_D\sigma,\quad c_D:=(\tfrac{\int_{B_D(0,1)} \varphi_1(\sqrt{2}u) \,du}{2\int_{B_D(0,1)} (u^{(1)})^2 \varphi_1(\sqrt{2}u)\, du})^{\frac{1}{2}},$$
then we can construct a local mean shift estimator
$F:\mathbb{R}^D\to\mathbb{R}^D$ as
\begin{equation}\label{eq:def:F(z)}
F(z):=\frac{\sum_{j=1}^{N}\varphi_r(y_j-z)y_j}{\sum_{j=1}^{N}\varphi_r(y_j-z)}. 
\end{equation}

\begin{figure}[htbp]
    \begin{tikzpicture}    
        \draw (0.2 ,0) node[above] {$\mathcal{M}$} arc (50:130:6); 
        
        \draw [dotted] (-8,1.4)--(-4,1.4) node[left] {$\pi(z)$}--(0,1.4)  node[above] {$T_{\pi(z)}\mathcal{M}$};
        \draw [dotted] (-4,2) node[above] {$z$} --(-4,1.0718);
        
        \fill [blue] (-4,2) circle (1pt);
        \fill [blue] (-4,1.4) circle (1pt);
        
        \fill [blue] (-4.05,1.2) node[below = 3pt] {$F(z)$} circle (1pt); 
        
        \draw [blue, densely dotted] (-4,2) circle (2);
        
        \draw [->, blue, densely dotted] (-4,2) -- (-5.7,3) node[below = 11pt, right = 3pt]{$r$};
        \draw [->, thick] (-4,1.4) -- (-4,2) node[below = 11pt, right = 3pt]{$v_z$};

        \fill (-2.9811,0.9126) circle (1pt); \fill (0.1380,0.7562) circle (1pt); \fill (-5.1115,0.0816) circle (1pt); \fill (-5.0900,0.5550) circle (1pt);
            \fill (-1.4464,0.1926) circle (1pt); \fill (-1.7007,1.3119) circle (1pt); \fill (-5.9252,0.2342) circle (1pt); \fill (-4.9937,0.3653) circle (1pt);
            \fill (-0.1001,0.2172) circle (1pt); \fill (-0.7597,0.9058) circle (1pt); \fill (-3.8094,2.1171) circle (1pt); \fill (-6.0577,0.9415) circle (1pt);
            \fill (-4.1842,0.5993) circle (1pt); \fill (-1.7753,0.7675) circle (1pt); \fill (-0.9972,0.9182) circle (1pt); \fill (-6.5370,1.6932) circle (1pt);
            \fill (-3.7693,1.4801) circle (1pt); \fill (-5.3435,0.6265) circle (1pt); \fill (-3.4319,0.7174) circle (1pt); \fill (-5.9580,0.5476) circle (1pt);
            \fill (-2.8541,1.2264) circle (1pt); \fill (-0.2795,0.4803) circle (1pt); \fill (-0.6471,0.6389) circle (1pt); \fill (-3.2643,0.6250) circle (1pt);
            \fill (-2.4899,0.6930) circle (1pt); \fill (-2.0460,1.2104) circle (1pt); \fill (-3.3017,2.2170) circle (1pt); \fill (-2.5288,0.4695) circle (1pt);
            \fill (-5.1595,0.9918) circle (1pt); \fill (-7.3842,0.3415) circle (1pt); \fill (0.0170,0.4563) circle (1pt);
            \fill (-7.3544,0.5188) circle (1pt); \fill (-5.9907,0.3842) circle (1pt); \fill (-5.4281,0.5091) circle (1pt); \fill (-6.2409,1.2324) circle (1pt);
            \fill (-3.7367,1.5424) circle (1pt); \fill (-5.4203,0.5675) circle (1pt);
    \end{tikzpicture}
    \caption{An illustration for constructing the estimator. The underlying submanifold $\mathcal{M}$, the estimator $F$ at point $z$ is constructed by the weighted average of noisy samples (black dots) lying within a distance $r$ from $z$. Let $\pi(z)$ denote the projection of $z$ onto $\mathcal{M}$, and define the vector $v_z = z - \pi(z)$.
    }\label{fig:estimatorF}
\end{figure}
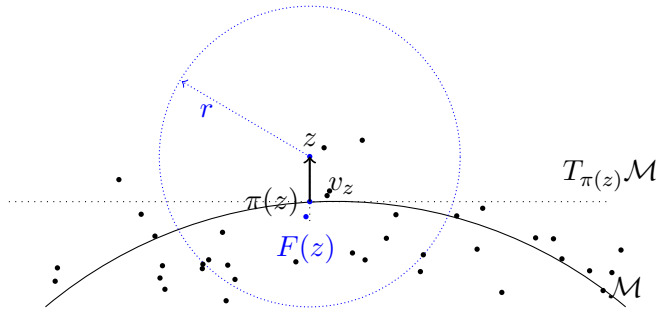

Since the Gaussian distribution can be approximated to vanish within a few standard deviations $C\sigma$, it is reasonable to consider the tubular neighborhood of $\mathcal{M}$ as
\[\Gamma = \{(x,v)\in T^\bot\mathcal{M} : \|v\|\le C\sigma\}.\]
We assume that 
the noise bandwidth $\sigma$ is sufficiently small compared to the reach $\tau$ of the manifold $\mathcal{M}$. Specifically, 
suppose $\sigma \le \sigma_0:=\frac{\tau-\varepsilon}{C+\sqrt{2}c_D}$ for any fixed $\varepsilon\in(0,\tau)$. 
\begin{thm}\label{thm:intro}
If the  sample size $N = O(\sigma^{-3d-5})$, then the estimator $F(z)$ satisfies
\[
F(z) = \pi(z) + \tfrac{d}{2} H_{\pi(z)}\sigma^2 + O(\sigma^3),
\quad \text{uniformly for all } z\in\Gamma,
\]
with probability at least $1 - C_1 \exp(-C_2 \sigma^{-c})$  
for some constants $c, C_1, C_2>0$, and 
where $H_{\pi(z)}$ is the mean curvature vector at projection $\pi(z)\in \mathcal{M}$.
The estimated $d$-dimensional manifold
$\widehat{\mathcal{M}}=F(\Gamma)$ has a 
Hausdorff distance $\tfrac{d}{2}H_{\pi(z)}\sigma^2$ to $\mathcal{M}$ and a reach bound below by $c'\tau$ for $c'>0$ with high probability.
\end{thm}
\begin{rem}
(1) The leading bias is $O(\sigma^2)$ in the normal direction, proportional to the mean curvature vector, while the tangential error is of order $O(\sigma^3)$. This reflects the quadratic normal deviation due to the mean curvature, whereas tangential errors cancel to higher order due to symmetry. 

(2) The error reduces to $O(\sigma^3)$ where the mean curvature vanishes, hence on any minimal submanifolds (where mean curvature vanishes everywhere, such as a Calabi–Yau manifold). This result establishes a novel statistical counterpart to the geometric prominence of minimal surfaces. It also improves upon the $O(\sigma^2)$ Hausdorff distance bound obtained in \cite{fefferman2021fitting} by Fefferman et al..  

\end{rem}


Our analysis further situates the estimator within several broader theoretical frameworks.

\textbf{Random walk on a graph.} The estimator $F$ can be viewed as one step of a random walk on the local graph constructed by observed points, where the Markov transition probabilities are given by the normalization of kernel $\varphi_r$. This process enhances the underlying geometric structure and suppresses noise-induced distortions, so that the leading term of $F(z)$ recovers the projection $\pi(z)$.

\textbf{Diffusion process.} The isotropic Gaussian noise model corresponds to a diffusion process on \(\mathbb{R}^D\) with heat kernel \(k_t(x,y)=(4\pi t)^{-D/2}e^{-\|x-y\|^2/(4t)}\) and \(\sigma^2=2t\). Hence the noisy density \(p_\sigma\) in \eqref{eq:def:nu} is obtained by diffusing the intrinsic density \(\omega\) on \(\mathcal{M}\) for time \(t=\sigma^2/2\). Unlike denoising schemes \cite{hein-maier-2006} that attempt to reverse this diffusion. Note that it is generally mathematically inconsistent since the graph Laplace operator is positive, and the evolution remains diffusive in nature. In contrast, our  estimator works directly with the diffused distribution. It suppresses noise by enhancing the underlying geometry, recovering the projection onto \(\mathcal{M}\) as the leading term, while the second‑order bias is explicitly governed by the mean curvature vector of \(\mathcal{M}\).

\textbf{Discrete mean‑curvature flow.} Theorem \ref{thm:intro} yields 
\[
\lim_{\sigma\to0^+} \frac{F(\pi(z)) - \pi(z)}{\sigma^2} =\tfrac{d}{2} H_{\pi(z)}.
\]
The operator $F-id$ on the left-hand side coincides with a normalized graph Laplacian built from truncated normalized weights $\varphi_r$. The limiting behavior indicates that the normalized graph Laplacian, when applied to the coordinate functions, converges to the mean curvature vector field on the manifold. This result overlaps with the theory of diffusion maps \cite{COIFMAN20065}. The estimator $F$ on $\mathcal{M}$ acts as a discrete mean curvature flow for a time scale of $\sigma^2$, enhancing both geometric and topological structure of the data.
In noisy settings, our analysis shows that the leading term of $F-id$ recovers the latent manifold, while the bias term arises from second-order extrinsic curvature.

\subsection{Main contribution}
Our approach fundamentally differs from previous manifold fitting techniques by circumventing the sequential error propagation that plagues multi-stage estimators. Conventional methods based on local PCA first construct a tangent space estimate \(\widehat{T_y\mathcal{M}}\) at a noisy point \(y\). This linear approximation inherently carries a normal-direction bias of order \(O(r^2)\), where \(r\) is the local scale parameter. This bias propagates into subsequent steps, such as estimation of the normal space  \cite{yao2019manifold}, construction of an approximate distance function \cite{mohammed2017manifold,fefferman2018fitting}, and the identification of the manifold as a zero-level set or ridge \cite{genovese2014nonparametric}. Each stage amplifies the initial geometric distortion, leading to a cumulative error that is difficult to characterize and control rigorously.

In contrast, by leveraging the explicit representation of \(p_\sigma\) through the exponential map on \(\mathcal{M}\), we perform a statistical inversion of the data-generating process. This direct analysis reveals that the state-of-the-art error \(O(\sigma^2)\) in the local mean shift is a precise geometric signal: it is proportional to the mean curvature vector \(H_{\pi(z)}\) of the underlying manifold at the projection point \(\pi(z)\). Thus, the mean curvature as a central quantity in extrinsic geometry emerges from the probabilistic structure of the corrupted data.

Beyond this theoretical insight, our method offers several practical advantages over existing manifold fitting frameworks:

\textbf{Theoretically grounded bandwidth selection.} The local averaging radius \(r = c_D \sigma\) is derived from an asymptotic analysis (Theorem~\ref{thm: loc average}, Theorem~\ref{th:score_noisy}), where the constant \(c_D\) depends explicitly on the ambient dimension \(D\) and the chosen kernel. This provides a rigorous, principled criterion for choosing the critical scale parameter. Many existing methods 
    lack such guidance; their performance hinges on heuristics for choosing a neighborhood size (\(r\) or \(k\) in \(k\)-NN), which often necessitates computationally expensive cross-validation or tuning. Our derived bandwidth not only ensures optimal bias-variance tradeoff for the stated theoretical guarantees but also translates into a computationally efficient and reproducible procedure.

\textbf{Minimal parameter requirements.} The construction requires knowledge of only the noise bandwidth \(\sigma\), which can often be estimated from the data or is known from the model. Crucially, it does not require prior knowledge of the intrinsic dimension \(d\) of the manifold. This is a practical benefit, as \(d\) is typically unknown in applications and its estimation is a separate, non-trivial problem. Methods relying on local PCA or \(k\)-NN graphs inherently depend on a correct choice of \(d\) for constructing the tangent space, making them sensitive to dimension misspecification.

\textbf{Explicit and non-iterative projection.} The estimator \(F(z)\) approximates the projection onto the manifold in a single, explicit mean-shift step. This contrasts with iterative schemes (e.g., subspace-constrained mean shift or gradient ascent on a density ridge) whose convergence properties and fixed points may not admit a clear geometric interpretation. It also contrasts with implicit representation methods that define the estimated manifold as the zero set of a learned function. Such implicit representations obscure the direct correspondence between an input point \(z\) and its projection, complicating both analysis and practical use.



\subsection{Organization}
This paper is organized as follows. Section 2 presents the geometric and probabilistic preliminaries.
 Section 3 analyzes the asymptotic properties of the population-level kernel estimator $F(z)$. 
Section 4 establishes high-probability quantitative bounds for the convergence of the sample estimator \(F(z)\) to its population counterpart \(\mu_z\). The geometric analysis of the output manifold is provided in Section 5. Numerical experiments demonstrating the method's efficacy are presented in Section 6. Finally, Section 7 provides a summary of all proofs.

\section{Preliminaries}
This section establishes the notations and reviews geometric and probabilistic facts used throughout the paper.

\subsection{Notations}
We fix integers $d, D$, and consider a compact $C^3$ submanifold $\mathcal{M}\subset \mathbb{R}^D$ with dimension $d<D$. The ambient space $\mathbb{R}^D$ is equipped with the Euclidean metric $g_E$ and the $L^2$-norm  $\|\cdot\|$. The induced Riemannian metric on $\mathcal{M}$ is denoted by $g:=g_E \mid _\mathcal{M}$. We write $\operatorname{Vol}(\mathcal{M})$ for the volume of $\mathcal{M}$ with respect to the induced volume measure $d\mu$ and assume $\operatorname{Vol}(\mathcal{M})< \infty$. For $m \in \mathbb{N}$, $z \in \mathbb{R}^m$ and $r > 0$ we denote by $B_m(z,r)$ the open Euclidean ball of radius $r$ centered at $z$ in $\mathbb{R}^m$. 

Throughout the paper $x$ represents a point on the latent manifold $\mathcal{M}$, $y$ represents a point related to the distribution $p_\sigma$, and $z$ represents an arbitrary point in the ambient space.
 Generic constants $C, C_1, C_2, \dots$ may change from line to line and, unless otherwise stated, depend only on the manifold $\mathcal{M}$, 
the reach $\tau$, and occasionally on the choice of $\varepsilon$. In particular, these constants are uniform in any fixed tubular neighborhood $T(\tau-\varepsilon)$.

\subsection{Geometric quantities of submanifolds}\label{se:Geometric_quantities}

We introduce several geometric quantities that characterize the embedding $\mathcal{M} \subset \mathbb{R}^D$.
We denote by $\nabla$ the ambient derivative in $\mathbb{R}^D$ (i.e. the Levi--Civita connection of $g_E$), and by $\nabla^\mathcal{M}$ the Levi--Civita connection on $(\mathcal{M},g)$. The tangent space and normal space at $x\in \mathcal{M}$ are denoted by $T_x\mathcal{M}$ and $T_x^\bot \mathcal{M}$ respectively. The tangent bundle and normal bundle are denoted by $T\mathcal{M}$ and $T^\bot \mathcal{M}$.

The \emph{second fundamental form} at $x \in \mathcal{M}$ is $\Pi_x : T_x\mathcal{M} \times T_x\mathcal{M} \to T_x^\perp\mathcal{M}$, defined for vector fields $X,Y$ on $\mathcal{M}$ by
\[
\Pi_x(X,Y) = (\nabla_X Y)^\perp,
\]
where $(\cdot)^\perp$ denotes orthogonal projection onto the normal space. This bilinear form measures the extrinsic curvature of $\mathcal{M}$. For a tensor $T$ we denote its operator norm by $\|T\|_{\text{op}}$; for instance,
\[
  \|\Pi_x\|_{\text{op}}
  := \sup_{\|u\|=\|v\|=1} \|\Pi_x(u,v)\|,
  \qquad u,v \in T_x \mathcal{M}.
\]

For each normal direction $n \in T_x^\perp\mathcal{M}$, the associated \emph{shape operator} $S_n : T_x\mathcal{M} \to T_x\mathcal{M}$ is defined by
\[
S_n(X) = -(\nabla_X n)^\top,
\]
where $(\cdot)^\top$ denotes the projection onto $T_x \mathcal{M}$. The second fundamental form and shape operator are related by the identity
\[
\langle S_n(X), Y \rangle = \langle \Pi(X,Y), n \rangle.
\]
The eigenvalues of $S_n$ are the \emph{principal curvatures} in the direction $n$, and the corresponding eigenvectors are the \emph{principal directions}.

The \emph{mean curvature vector} $H_x \in T_x^\perp\mathcal{M}$ provides a global measure of bending:
\begin{equation}\label{eq:meancurvature}
H_x = \frac{1}{d} \operatorname{Tr}(\Pi_x) = \frac{1}{d} \sum_{i=1}^d \Pi_x(e_i,e_i),
\end{equation}
where $\{e_i\}$ is an orthonormal basis of $T_x\mathcal{M}$. In particular, $\langle H, n \rangle$ gives the mean curvature in the direction $n$.
Under \emph{mean curvature flow}, the manifold evolves proportional to its mean curvature.  Minimal submanifolds have zero mean curvature. 

The \emph{reach} $\tau$ of $\mathcal{M}$ controls these curvature quantities globally. 
\begin{definition}[Reach] \label{def:reach}
    Let $A$ be a closed subset of $\mathbb{R}^D$. The reach of $A$, denoted by reach($A$), is the largest number $\tau$ satisfying: any point at a distance less than $\tau$ from $A$ has a unique nearest point in $A$.
\end{definition}
We assume throughout that $\mathcal{M}$ has positive reach $\tau > 0$.
Intuitively, a large reach implies that the manifold is locally close to its tangent space, which can be explained in the following lemma \cite{federer1959curvature}.
\begin{lem}[Federer's reach condition]\label{Lemma:ReachCond}
Let $\mathcal{M}$ be an embedded sub-manifold of $\mathbb{R}^{D}$ with reach $\tau>0$. Then
    $$\tau^{-1} = \sup \left\{\frac{2d(b,T_a\mathcal{M})}{\|a-b\|^2}|a,b\in\mathcal{M},~a\neq b\right\}.$$
\end{lem}

Positive reach implies a uniform bound on the second fundamental form,
and together with the regularity and compactness of $\mathcal{M}$
yields uniform bounds on curvature tensors and their covariant derivatives.

\begin{lem}[
\cite{Niyogi2008FindingTH}, Proposition 6.1]\label{lem:second fundamental form bound}
If $\mathcal{M} \subset \mathbb{R}^D$ has positive reach $\tau > 0$, then
\begin{align}
  \|\Pi\|_{\text{op}} \le \tau^{-1}.
\end{align}
In particular, when $\mathcal{M}$ is compact and $C^3$, all curvature tensors and the second fundamental form admit bounded covariant derivatives.
\end{lem}
This implies that all principal curvatures are bounded by $\tau^{-1}$ in magnitude, ensuring the manifold does not curve too sharply. 

\subsection{Exponential map}

 The \emph{exponential map} at $x \in \mathcal{M}$ is denoted by $\exp_x: T_x \mathcal{M} \to \mathcal{M}$. The \emph{injectivity radius} at $x \in \mathcal{M}$ is denoted by $\text{inj}(x)$, which is the largest radius $r>0$ such that $\exp_x:B_d(0,r)\to \mathcal{M}$ is a diffeomorphism. The global injectivity radius of $\mathcal{M}$ is defined as $\text{inj}_\mathcal{M}=\inf_{x\in \mathcal{M}} \text{inj}(x)$. With positive reach $\tau>0$, there is a lower bound for the injectivity radius.

\begin{lem}[~\cite{Alexander2005GaussEA} Corollary 1.4]\label{lem:injectivity radius} 
    Let $\mathcal{M}\subset\mathbb{R}^D$ be a $C^2$ manifold with positive reach $\tau>0$, then
    \[
    \text{inj}_\mathcal{M} \ge \pi\tau.
    \]
\end{lem}

For $r < \text{inj}(x)$, the exponential map introduces a exponential coordinate chart $  \bigl( \exp_x(B_d(0,r)), \exp_x^{-1} \bigr)$, which provides the following standard expansions for a local coordinate system around $x$. 
 
\begin{lem}\label{lem:exponential map expansion}
     For $x \in \mathcal{M}$ and $u\in T_x\mathcal{M}$ with $\|u\|$ sufficiently small, the exponential map satisfies
     \[
     \text{exp}_x(u) = x  + u +\frac{1}{2}\Pi_x(u,u) + O(\|u\|^3). 
     \]
     where $\Pi_x : T_x M \times T_x M \to T_x^\perp M$ is the second fundamental form at $x$.
\end{lem}

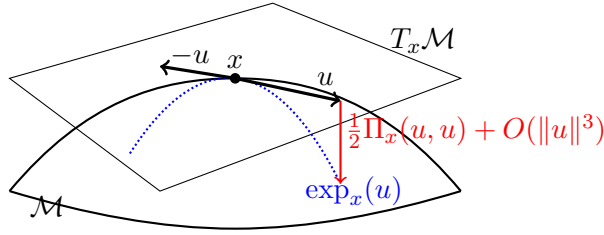
\begin{figure}[htbp]
\begin{tikzpicture}

\draw[fill=white, fill opacity=0.8] 
  (-2.5,1.5) -- (1,2.5) -- (3.5,1.5) -- (-0.5,0) -- cycle;

\node at (3,2) {$T_x\mathcal{M}$};

\draw[thick] 
  (-2.5,0) .. controls (-1.5,1.2) and (-0.5,1.5) .. (0.5,1.5)
  .. controls (1.5,1.5) and (2.5,1.2) .. (3.5,0);

\draw[thick] 
  (-2.5,0) .. controls (-1.5,-0.3) and (-0.5,-0.5) .. (0.5,-0.5)
  .. controls (1.5,-0.5) and (2.5,-0.3) .. (3.5,0);

\node at (-2,-0.2) {$\mathcal{M}$};

\fill (0.5,1.5) circle (2pt);
\node[above] at (0.5,1.5) {$x$};

\draw[densely dotted, blue, thick] 
  (-0.9,0.5) .. controls (0,1.9) and (1,1.9) .. (1.9,0.1);

\draw[->, very thick] 
  (0.5,1.5) -- (-0.5,1.66);
\node at (-0.1,1.8) {$-u$};
\draw[->, very thick] 
  (0.5,1.5) -- (1.9,1.2);
\node at (1.7,1.5) {$u$};

\draw[->, red, thick] 
  (1.9,1.2) -- (1.9,0.1);
\node[red] at (3.7,0.8) {$\frac{1}{2}\Pi_x(u,u) + O(\|u\|^3)$};

\node[blue, right] at (1.3,0) {$\exp_x(u)$};

\end{tikzpicture}
\caption{Exponential map and its Taylor expansions.}
\end{figure}

\begin{lem}\label{lem:volume element expansion}
For $x\in \mathcal{M}$, let $\{u_i\}_{i=1}^d$ be an orthonormal basis of $T_x\mathcal{M}$, and denote by $g_{ij}(u)$ the components of the metric in exponential coordinates around $x$. Then as $\|u\| \to 0$,
\[
\sqrt{\text{det} g_{ij}(u)} = 1 + O(\|u\|^2).
\]
\end{lem}

\subsection{Hausdorff distance and tubular neighborhoods}
The distance between a point $a$ and a non-empty set $A$ is represented as $d(a,A) = \inf_{a^\prime\in A} \|a-a^\prime\|$. 
We adopt the \emph{Hausdorff distance} to measure the distance between two sets. 
\begin{definition}[Hausdorff distance] 
Let $ A_1$ and $ A_2$ be two non-empty subsets of $\mathbb{R}^D$. We define their Hausdorff distance $\operatorname{dist}_H(A_1,A_2)$ as
    $$
    \operatorname{dist}_H(A_1, A_2) = \max \{\sup_{a\in  A_1} \inf_{b\in  A_2}\|a-b\|,~\sup_{b\in  A_2} \inf_{a\in  A_1}\|a-b\|\}.
    $$
\end{definition}

For $\mathcal{M}$ with positive reach $\tau > 0$,
define the open $\tau$-tubular neighborhood of $\mathcal{M}$ in the normal bundle by
\[
  \mathcal{T}(\tau) := \{ (x,v) \in T^\perp \mathcal{M} : \|v\| < \tau \}.
\]
We identify $(x,v) \in \mathcal{T}(\tau)$ with the point $x + v \in \mathbb{R}^D$. For $0<\varepsilon<\tau$ we also introduce the closed tubular neighborhood
\[
  T(\varepsilon) := \{ (x,v) \in T^\perp \mathcal{M} : \|v\| \le \varepsilon \}.
\]
Since $\mathcal{M}$ is compact and $\tau > 0$, the set $T(\varepsilon)$ is compact.


For a non-empty closed set $A \subset \mathbb{R}^D$,  note that the nearest point to $A$ is unique and we denote the corresponding projection point by $\pi_A(z) := \arg \min_{a \in A} \|z - a\|$. By default, we write $\pi(z)$ for the nearest-point projection onto $\mathcal{M}$ whenever $d(z,\mathcal{M}) < \tau$.

The positive reach assumption ensures that the projection $\pi : \mathcal T(\tau) \to \mathcal{M}$ is well-defined. Since $\mathcal{M}$ is $C^3$, the tubular neighborhood theorem implies that $\pi$ is $C^{2}$ on $\mathcal T(\tau)$. In particular, for every $0 < \varepsilon < \tau$ and $m=0,1$ there exist constants $B_{\pi,\varepsilon,m} > 0$ such that
\[
  \sup_{y \in T(\varepsilon)} \|\nabla^m \pi(y)\|
  \le B_{\pi,\varepsilon,m}.
\]
Consequently, the normal vector $v_y := y - \pi(y)$ is $C^{2}$ in $y$, and its derivatives $\nabla v_y$ are uniformly bounded on $T(\tau -\varepsilon)$.

We will repeatedly use a simple Lipschitz-type estimate for tensor fields
on tubular neighborhoods.

\begin{lem}\label{lem:Lip-derivative}
Let $\Gamma \subset T(\tau)$ be compact and let $S$ be a $C^{m+1}$ tensor field
defined on an open neighborhood of $\Gamma$. Then $\nabla^m S$ is Lipschitz on
$\Gamma$: there exists a constant $C > 0$, depending only on $m$ and the upper bound of $\nabla^m S$, such that for all $x,y \in \Gamma$,
\[
  \bigl\|\nabla^m S(y) - \nabla^m S(x)\bigr\|_{\mathrm{op}} \le C \|y - x\|.
\]
In particular, taking $\Gamma = T(\tau - \varepsilon)$ and $x = \pi(y)$ we obtain
\[
  \bigl\|\nabla^m S(y) - \nabla^m S(\pi(y))\bigr\|_{\mathrm{op}}
  \le C \|y - \pi(y)\|
  = C \|v_y\|,
  \quad y \in T(\tau - \varepsilon).
\]
\end{lem}

\subsection{Probability preliminaries}
We first summarize that probability measures are denoted by $\bP$ (for a generic distribution), by $\bP_\mathcal{M}$ for a distribution supported on $\mathcal{M}$, and by $p_\sigma$ for the noisy distribution with noise level $\sigma > 0$. Expectations 
are written as $\mathbb{E}(\cdot)$. 

Recall the generic Chernoff bound for a random variable $X$ is attained by applying Markov's inequality to $e^{tX}$.
\begin{lem}[Chernoff bound \cite{chernoff1952measure}]\label{lem:Chernoff bound}
    If $X$ is a random variable, then for every $t>0$
    $$\bP(X \geq a)=\bP\left(e^{t X} \geq e^{t a}\right) \leq \frac{\bE\left(e^{t X}\right)}{e^{t a}}.$$
\end{lem}  
Note that the Hoeffding's inequality provides an upper bound on the probability that the average of independent random variables deviates from its expected value.
\begin{lem}[Hoeffding's inequality {\cite{hoeffding1963probability}}]
\label{lem:Hoeffding}
Let $X_1,\dots,X_n$ be independent real-valued random variables satisfying $a_i\le X_i\le b_i$ for each $i$. Set $\bar X_n=\frac{1}{n}\sum_{i=1}^n X_i$. Then for every $t>0$,
\[
\bP\big(\bar X_n - \mathbb{E}(\bar X_n) \ge t\big)
\le \exp\!\Big(-\frac{2n^2 t^2}{\sum_{i=1}^n (b_i-a_i)^2}\Big).
\]
In particular, if $b_i-a_i\le L$ for all $i$ then
\[
\bP\big(|\bar X_n - \mathbb{E}(\bar X_n)|\ge t\big)
\le 2\exp\Big(-\frac{2n t^2}{L^2}\Big).
\]
\end{lem}

Hence we prove the following lemma to control the deviation between the bounded weighted average of independent random vectors and its expectation. 
\begin{lem}
\label{lem:nonasymp_hoeffding}
Let $\{Y_i\}_{i=1}^n$ be i.i.d. random vectors in $\mathbb{R}^D$ and let $W:\mathbb{R}^D\to[0,\infty)$ be a measurable weight function. Define
\[
S_n:=\frac{1}{n}\sum_{i=1}^n W(Y_i)Y_i,\qquad
B_n:=\frac{1}{n}\sum_{i=1}^n W(Y_i),
\]
and set $s:=\mathbb{E}(W(Y)Y)$, $b:=\mathbb{E}(W(Y))>0$, $\widehat\mu_n := S_n / B_n$ and $\widehat\mu_w=s/b$. Assume that
\begin{enumerate}
  \item $\{Y_i:W(Y_i)>0\}\subset B_D(z,r)$ for all $i$ and some $z\in\mathbb{R}^D$ and $r>0$;
  \item $0\le W(Y_i)\le M$ for all $i$;
  \item there exists $K>0$ such that for every coordinate $j=1,\dots,D$ and every $i$,
  \[
  \big|(W(Y_i)Y_i)^{(j)} - \mathbb{E}((W(Y)Y)^{(j)})\big|\le K.
  \]
\end{enumerate}
Then for any $\varepsilon>0$, we obtain
\begin{equation}\label{eq:combined_tail}
\bP\big(\|\widehat\mu_n-\widehat\mu_w\| > \varepsilon\big)
\le 2\exp\!\Big(-\frac{n b^2}{2M^2}\Big)
+ 2D\exp\!\Big(-\frac{n b^2\varepsilon^2}{32 D K^2}\Big)
+ 2\exp\!\Big(-\frac{n b^4\varepsilon^2}{8\|s\|^2 M^2}\Big).
\end{equation}

\end{lem}

\section{Normalized local kernel estimator}\label{sec:Local normalized kernel estimator}
This section analyzes the asymptotic properties of the local kernel estimator \eqref{eq:def:F(z)} at the population level. We derive expansions for the population density $p_\sigma$ and the corresponding score function $\nabla\log p_\sigma(z)$, which encode the curvature information of the underlying manifold. 
Finally, we prove that the population-level estimator yields an estimated manifold sufficiently close to the underlying manifold.

Recall that for any $r> 0$, let $\varphi_r:\mathbb{R}^D\to(0,\infty)$ be  a smooth compact supported rotation-invariant kernel with bandwidth $r$, for example, a truncated Gaussian kernel
\begin{equation}\label{eq:truncated Gaussian kernel-se3}
    \varphi_r(z) =\phi_r (z) \chi\!\big(\tfrac{\|z\|}{\sqrt{2}r}\big)=({2\pi r^2})^{-\frac{D}{2}}e^{\frac{-\|z\|^2}{2r^2}}\,\chi\!\big(\tfrac{\|z\|}{\sqrt{2}r}\big),
\end{equation}
where \(\chi:[0,\infty)\to[0,1]\) is a smooth cutoff function satisfying
\[
\chi(t)=1\ \text{ for } 0\le t\le \rho_0<1,\qquad
\chi(t)=0\ \text{ for } t\ge 1.
\]
Then $\varphi_r\in C_0^{\infty}(B_D(0,\sqrt{2}r))$. Based on observed samples $\{y_i\}_{i=1}^N$ as defined in \eqref{eq:Add_model}, we can construct a local kernel estimator
$F:\mathbb{R}^D\to\mathbb{R}^D$ as
\begin{equation}\label{eq:def:F(z)-se3}
F(z):=\frac{\sum_{j=1}^{N}\varphi_r(y_j-z)y_j}{\sum_{j=1}^{N}\varphi_r(y_j-z)}.
\end{equation}


\subsection{Asymptotic analysis for the population-level estimator}
As the sample size increases, the sample-level normalized estimator converges to its population-level estimator 
$${\mu}_z :=\frac {\bE(\varphi_r(Y-z)Y)}{\bE(\varphi_r(Y-z))}=\frac{\int_{B_D(z,\sqrt{2}r)} \varphi_r(y-z)\,y\,p_\sigma(y)\,dy}
           {\int_{B_D(z,\sqrt{2}r)} \varphi_r(y-z)\,p_\sigma(y)\,dy}$$ 
with high probability, see details in Section 4. 

First, we consider the asymptotic behavior of its expectation with respect to the truncation radius $r$.

\begin{thm}\label{thm: loc average}
For a smooth density $p_\sigma:\mathbb{R}^D \to \mathbb{R}$, the expectation of the estimator $F:\mathbb{R}^D\to\mathbb{R}^D$ defined in \eqref{eq:def:F(z)-se3} satisfies the asymptotics:
\begin{equation}\label{eq:loc average}
    \mu_z=\frac{\int_{B_D(z,\sqrt{2}r)} \varphi_r(y-z)\,y\,p_\sigma(y)\,dy}
           {\int_{B_D(z,\sqrt{2}r)} \varphi_r(y-z)\,p_\sigma(y)\,dy}
    = z + r^2\,\frac{2B\,\nabla p_\sigma(z)}{A\,p_\sigma(z)} + O(r^4),
\end{equation}
where $A = \int_{B_D(0,1)} \varphi_1(\sqrt{2}u) \,du$ and $B = \!\int_{B_D(0,1)} (u^{(1)})^2 \varphi_1(\sqrt{2}u)\, du$ are constants depending on dimension $D$.
\end{thm}
\begin{rem}
    According to the bandwidth $\sigma$ of noise, taking $r=c_D\sigma=(\frac{A}{2B})^{\frac{1}{2}}\sigma$ 
    we obtain
\[ \mu_z=z + \sigma^2\,\nabla\log p_\sigma(z) + O(\sigma^4).\]  
The leading error term is governed by the score function $\nabla\log p_\sigma(z)$.

\end{rem}

The key of the proof lies in the local Taylor expansion of the density function, combined with the radial invariance of both the local kernel and the integration region. See details in Section 7.

\subsection{Density of noise samples on manifold}

Recall that $\mathcal{M} \subset \mathbb{R}^D$ is a compact, $C^3$-smooth $d$-dimensional submanifold endowed with the induced volume measure $d\mu$. Its volume is $\operatorname{Vol}(\mathcal{M})=\int_\mathcal{M}d\mu<\infty$. We consider the uniform probability distribution on $\mathcal{M}$ with density 
\[
\omega(x) = \frac{1}{\operatorname{Vol}(\mathcal{M})}, \;x\in \mathcal{M}.
\]
Adding isotropic Gaussian noise $\phi_\sigma\sim \mathcal{N}(0,\sigma^2 I_D)$ yields a probability distribution on $\mathbb{R^D}$ defined by the convolution
\begin{align}
\notag&p_{\sigma}(y)= \int_\mathcal{M} \phi_\sigma(y-x)\omega(x)d\mu(x) = \frac{1}{\operatorname{Vol}(\mathcal{M})(2\pi\sigma^2)^{\frac{D}{2}}} \int_\mathcal{M} e^{-\frac{\|y-x\|^2}{2\sigma^2}}d\mu(x). 
\end{align}
The closed tubular neighborhood of $\mathcal{M}$ with radius $\tau-\varepsilon>0$ is $$T(\tau-\varepsilon) =\{(x,v)\in T^\bot \mathcal{M}\mid\|v\|\le\tau-\varepsilon\}.$$
For point $y\in T(\tau)$, let $\pi(y)$ denote its projection onto $\mathcal{M}$ and $v_y = y-\pi(y)$ the corresponding normal displacement. The following theorem gives an expansion of $p_\sigma(y)$ in terms of $\|v_y\|$ and $\sigma$.

\begin{thm}\label{th:probability noisy}
Let $\mathcal{M}\subset\mathbb{R}^D$ be a compact $C^3$ submanifold of dimension $d$ with reach $\tau>0$.
For any $\varepsilon\in(0,\tau)$ there exists constant $\sigma_0>0$ such that for all $y\in T(\tau-\varepsilon)$ and $0<\sigma\le\sigma_0$,
\begin{equation}\label{eq:probability noisy}
 p_\sigma(y)
    = \frac{1}{\operatorname{Vol}(\mathcal{M}) (2\pi\sigma^2)^{\frac{D-d}{2}}}
    e^{-\frac{\|v_y\|^2}{2\sigma^2}}
    \frac{1}{\sqrt{\det A_y}}
      \left(1 + O(\sigma\|v_y\| + \sigma^2)\right),
\end{equation}
where $A_y = I_d - \langle v_y, \Pi_{\pi(y)} \rangle$ and $\Pi_{\pi(y)}$ is the second fundamental form of $\mathcal{M}$ at the projection $\pi(y)\in\mathcal{M}$. The $O(\cdot)$ term is uniform in $y\in T(\tau-\varepsilon)$ and $\sigma\in(0,\sigma_0]$.
\end{thm}

The proof is based on the Taylor expansions of the exponential map and the volume element in exponential coordinates around $\pi(y)$ (Lemma~\ref{lem:exponential map expansion} and \ref{lem:volume element expansion}).  The noisy density factorizes into a Gaussian term in the normal direction $\exp\!\Big(-\frac{\|v_y\|^2}{2\sigma^2}\Big)$, and a curvature-dependent correction $(\det A_y)^{-1/2}$, up to a remainder of order $O(\sigma \|v_y\| + \sigma^2)$. Thus the local behavior of $p_\sigma$ is governed jointly by the normal distance $\|v_y\|$ and the second fundamental form $\Pi_{\pi(y)}$ at the projection point.

For the analysis of derivatives, it is convenient to rewrite Theorem~\ref{th:probability noisy} at the level of the log-density and to record uniform bounds on the remainder and its derivatives.
\begin{cor}~\label{co:log density}
Assume $\mathcal{M}$ is a compact $C^3$ submanifold with positive reach $\tau>0$.
For any $\varepsilon\in(0,\tau)$, there exist constants $\sigma_0>0$ and $C_{\varepsilon,m}>0$ ($m=0,1$) such that for all
$y\in T(\tau-\varepsilon)$ and $0<\sigma\le\sigma_0$,
\begin{equation}~\label{eq:log density}
\log p_\sigma(y)
=
\log \bigl(\operatorname{Vol}(\mathcal{M})(2\pi\sigma^2)^{\frac{D-d}{2}}\bigl)^{-1}
-
\frac{\|v_y\|^2}{2\sigma^2}
-
\frac{1}{2}\log\det A_y
+
R(y,\sigma),
\end{equation}
where $R$ satisfies
\[
|R(y,\sigma)| \le C_{\varepsilon,0}\bigl(\sigma\|v_y\|+\sigma^2\bigr).
\]
And the ambient derivatives of $R$ with respect to $y$
satisfy
\[
\|\nabla_y R(y,\sigma)\|_{\text{op}}
\;\le\;
C_{\varepsilon,1}\sigma.
\]
\end{cor}
Note that the log-density admits an expansion whose remainder is $O(\sigma\|v_y\|+\sigma^2)$, while the differentials with respect to $y$ only preserve an order of $O(\sigma)$.

\subsection{Score expansion}
Next, we want to derive an expansion for the score function
\[
s_\sigma(y) = \nabla \log p_\sigma(y), \qquad y \in T(\tau-\varepsilon).
\]
We first record the derivatives of $\pi(y)$ and $v_y$ with respect to the ambient coordinate $y$. Let $\nabla_\omega(\cdot)$ denote the Euclidean directional derivative in the direction $\omega \in {R}^D$, and let $\omega^\top$ denote the orthogonal projection of $\omega$ onto $T_{\pi(y)}M$.

\begin{lem}\label{lem:derivative of v pi}
    Let $\mathcal{M}\subset\mathbb{R}^D$ be a compact $C^3$ submanifold of dimension $d$ with reach $\tau>0$.
For any $y\in \mathcal T(\tau)$, $\omega\in \mathbb{R}^D$, the derivatives of $\pi(y)$ and $v_y$ follow
    \[
    \nabla_\omega\pi(y) = A_y^{-1} \omega^\top,\qquad\nabla_\omega v_y =\omega - A_y^{-1}\omega^\top.
    \]
\end{lem}

The proof uses the characterization of $A_y$ in terms of the second fundamental form. 
Combining Lemma ~\ref{lem:derivative of v pi} with the log-density expansion in Corollary ~\ref{co:log density}, we obtain the following expression for the score $\nabla\log p_\sigma(y)$.

\begin{thm}\label{th:score_noisy}
Let $\mathcal{M}\subset\mathbb{R}^D$ be a compact $C^3$ submanifold of dimension $d$ with reach $\tau>0$.
For any $\varepsilon\in(0,\tau)$ there exists constant $\sigma_0>0$ such that for all $y\in T(\tau-\varepsilon)$ and $0<\sigma\le\sigma_0$,
\begin{equation}\label{eq:score}
    s_\sigma(y):=\nabla\log p_\sigma(y) = -\frac{v_y}{\sigma^2} + \frac{d}{2}\,H_{\pi(y)} +O\!\big(\|v_y\|+ \sigma \|v_y\|+\sigma^2\big),
\end{equation}
where $H_{\pi(y)}=\frac{1}{d}\operatorname{Tr}\!\big(\Pi_{\pi(y)}\big)$ denotes the mean curvature at $\pi(y)\in \mathcal{M}$. The $O(\cdot)$ term is uniform in $y\in T(\tau-\varepsilon)$ and $\sigma\in(0,\sigma_0]$.
\end{thm}

\begin{rem}
(1) The leading term $-{v_y}/{\sigma^2}$ is induced by the Gaussian noise and points towards the manifold along the normal direction. The term $\frac{d}{2}H_{\pi(y)}$ captures the contribution of the mean curvature, arising from the determinant of $\langle v,\Pi\rangle$ in $p_\sigma$. The remainder $O(\|v_y\|+\|v_y\|\sigma+\sigma^2)$ vanishes as $y$ approaches $\mathcal{M}$ and $\sigma$ tends to zero. 

(2) For points within the tubular neighborhood, the tangential component of $s_\sigma(y)$ is negligible compared to its normal component with sufficiently small noise. This shows that the score field is essentially a normal vector field in $T(\tau-\varepsilon)$.

(3) It follows from \eqref{eq:score} that a point \(y\) satisfying \(v_y = \sigma^2 \frac{d}{2} H_{\pi(y)}\) corresponds to a local maximum of the density along the normal direction. This yields a precise geometric characterization of the ridge set within ridge theory \cite{genovese2014nonparametric}.
\end{rem}


\subsection{Population-level output manifold}
Next, combining Theorem ~\ref{thm: loc average} with the score $\nabla\log p_\sigma(z)$ expansion in Theorem~\ref{th:score_noisy} we are ready to prove that our population-level estimator will output a manifold that is $O(\sigma^2)$ close to the underlying manifold $\mathcal M$.
Recall that 
the Hausdorff distance between two non-empty subsets 
$A_1, A_2\subset\mathbb R^D$ is defined by
$$\operatorname{dist}_H(A_1, A_2) = \max \{\sup_{a\in  A_1} \inf_{b\in  A_2}\|a-b\|,~\sup_{b\in  A_2} \inf_{a\in  A_1}\|a-b\|\}.$$

\begin{thm}
\label{thm:Population-level output manifold}
Let $\mathcal M\subset\mathbb R^D$ be a compact $C^3$ submanifold of dimension $d$
with reach $\tau>0$.  For any $\varepsilon\in(0,\tau)$ there exists constant $\sigma_0>0$ such that for the population mapping $\mu:T(\tau-\varepsilon)\to\mathbb R^D$, $z\mapsto\mu_z$ defined in \eqref{eq:loc average}, and all $0<\sigma\le\sigma_0$,
\begin{equation}
\mu_z=\pi(z)+\frac{d}{2}H_{\pi(z)}\sigma^2+O(\sigma^2),
\end{equation}
where $H_{\pi(z)}$ denotes the mean curvature 
at 
$\pi(z)\in\mathcal M$. Moreover, the 
output manifold
\[
\mu(T(\tau-\varepsilon))=\{\mu_z:\,z\in T(\tau-\varepsilon)\}
\]
satisfies
\[
\operatorname{dist}_H(\mu(T(\tau-\varepsilon)),\mathcal M)\le C \sigma^2.
\]
\end{thm}

\begin{rem}

Although the original manifold $\mathcal M$ is only $C^3$, the population-level estimator $z\mapsto\mu_z$ inherits \emph{infinite smoothness} from kernel $\varphi_r$, and the image $\mu(T(\tau-\varepsilon))$ is a $C^\infty$ manifold.

\end{rem}

\section{Sample-level approximation and sample complexity}

In this section we establish quantitative bounds ensuring that the sample-level
estimator \(F(z)\) converges to its population counterpart \(\mu_z\) with high
probability. Using concentration inequalities for bounded random variables,
we determine the sample-size scaling under which the stochastic deviation
becomes negligible compared with the deterministic bias \(O(\sigma^2)\)
obtained in the population-level analysis.

Recall that $\mathcal{M}$ is a compact, $C^3$-smooth $d$-dimensional submanifold of $\mathbb{R}^D$ with finite volume $\operatorname{Vol}(\mathcal{M})$ and positive reach $\tau$. Based on the Gaussian noisy sample $\mathcal Y = \{y_i\}_{i=1}^N$ defined in \eqref{eq:Add_model}, 
we construct a local kernel estimator
\begin{equation*}
F(z):=\frac{\sum_{j=1}^{N}\varphi_r(y_j-z)y_j}{\sum_{j=1}^{N}\varphi_r(y_j-z)},\;r:=c_D\sigma,
\end{equation*}
where $\varphi_r\in C_c^{\infty} (B_D(0,\sqrt{2}r))$.
Let $\mathcal Y_z = \{y_i\}_{i=1}^n$ ($n\le N$) be the sample that plays a role in $F(z)$, i.e., the sample that falls within the local ball $B_D(z,\sqrt{2}r)$.
And denote its population-level estimator by
$${\mu}_z :=\frac {\bE(\varphi_r(Y-z)Y)}{\bE(\varphi_r(Y-z))}.$$ 
We consistently work within the closed the tubular neighborhood of $\mathcal{M}$ as
\[\Gamma = \{(x,v)\in T^\bot\mathcal{M} : \|v\|\le C\sigma\}.\]
Assume that 
the noise bandwidth $\sigma$ is sufficiently small compared to the reach $\tau$ of 
$\mathcal{M}$. Specifically, 
suppose $\sigma \le \sigma_0:=\frac{\tau-\varepsilon}{C+\sqrt{2}c_D}$ for any fixed $\varepsilon\in(0,\tau)$, which ensures $B_D(z,\sqrt{2}r)\subset T(\tau-\varepsilon)$ for obtaining the uniform estimation in Section 3.


\subsection{Local sample complexity}
Since the weights in the estimator $F$ are bounded by the truncated kernel, we can apply Lemma \ref{lem:nonasymp_hoeffding} to obtain
the following theorem that states that the local sample size $n \ge O(\sigma^{-2d-5})$ can control the deviation between the sample-level and population-level estimators with high probability. 

\begin{thm}
\label{thm:Concentrate_with_n}
For any $z\in \Gamma$, let  $\mathcal Y_z=\{y_i\}_{i=1}^n$ be the noisy points observed in the local ball $B_D(z,\sqrt{2}r)$ with $r=c_D\sigma$ and $0<\sigma\le\sigma_0$. Then for $n \ge O(\sigma^{-2d-5})$, one has 
    $$\bP(\|F(z) - \mu_z\|\leq C_3\sigma^{3}) \geq 1 - C_1\exp({-C_2 \sigma^{-1}}),$$
    for some constants $C_1$, $C_2$ and $C_3>0$.
\end{thm}


\subsection{Total sample complexity}
Next, by analyzing the density function \(p_\sigma\) we estimate the probability that the observation \(Y \sim  p_\sigma\) as defined in \eqref{eq:def:nu} falls in a local ball $B_D(z,\sqrt{2}r)$, and further derive the required total sample size $N$.
\begin{lem}
\label{Lemma:prob_in_a_ball}
For any $z\in \Gamma$ satisfying $B_D(z,\sqrt{2}r)\subset T(\tau-\varepsilon)$ with $r=c_D\sigma$ and $0<\sigma\le\sigma_0$, we obtain
\begin{equation}
    \mathbb{P}(Y\in B_D(z,\sqrt{2}r)) = c\, r^d(1+O(r)),
\end{equation}
where $Y$ is the observation defined in \eqref{eq:Add_model}. 
\end{lem}

\begin{cor}
\label{Col:Local_sample_size}
    Let $n$ be the number of observed points $Y$ defined in \eqref{eq:Add_model} that fall in $B_D(z,\sqrt{2}r)$ for $z\in \Gamma$. Assume $r=c_D\sigma$ and $0<\sigma\le\sigma_0$. If the total sample size satisfies $N\ge O(\sigma^{-5-3d})$, then
    $$\bP( n\ge C_1 \sigma^{-5-2d} )\geq 1 - 2\exp({ -C_2\sigma^{-5-2d}})$$ 
    for some constants $C_1$, $C_2>0$.
\end{cor}

\section{Sample-level estimator and output manifold}

For data points on a compact $C^3$ submanifold $\mathcal{M}\subset\mathbb{R}^D$ observed with Gaussian noise $\mathcal Y=\{y_i\}_{i=1}^N$ as defined in~\eqref{eq:Add_model}, this section establishes the asymptotic properties of the sample-level estimator $F$
defined in \eqref{eq:def:F(z)}. 
Consider the tubular neighborhood 
\[\Gamma = \{(x,v)\in T^\bot\mathcal{M} : \|v\|\le C\sigma\}.\]
Assume $\sigma \le \sigma_0:=\frac{\tau-\varepsilon}{C+\sqrt{2}c_D}$ for any fixed $\varepsilon\in(0,\tau)$ and $r=c_D\sigma$, which ensures $B_D(z,\sqrt{2}r)\subset T(\tau-\varepsilon)$. We demonstrate that for a sample size $N \ge O(\sigma^{-3d-5})$ the estimator $F$ effectively denoises points in $\Gamma$, mapping them to a vicinity of the underlying manifold $\mathcal{M}$. Moreover, we can characterize the geometry of the output manifold $F(\Gamma)$.



\subsection{Sample-level asymptotic consistency}
The following theorem establishes a pointwise estimate for the estimator $F(z)$ on $\Gamma$ with sample size $N \ge O(\sigma^{-3d-5})$.
\begin{thm}
\label{thm:main}
If the sample size satisfies $N \ge O(\sigma^{-3d-5})$, then the estimator $F(z)$ for any 
$z\in \Gamma$ as defined in \eqref{eq:def:F(z)} satisfies
\[
F(z) = \pi(z) + \tfrac{d}{2} H_{\pi(z)} \sigma^2 + O(\sigma^3)
\]
with probability at least $1 - C_1 \exp(-C_2 \sigma^{-1})$  
for some constants $ C_1, C_2>0$, and 
here $H_{\pi(z)}$ is the mean curvature at $\pi(z)\in \mathcal{M}$.
\end{thm}

\begin{rem}
    The above theorem indicates that the estimation accuracy is intrinsically governed by the geometric structure of the manifold.  
In particular, if the mean curvature $H_{\pi(z)}$ at $\pi(z)\in \mathcal{M}$ vanishes, the estimation error is reduced to $O(\sigma^3)$.
\end{rem}

\begin{cor}\label{cor:uniform estimate}
Under the assumptions of Theorem~\ref{thm:main}, let $\mathcal Y^o= \{y_i\}_{i=1}^n := \mathcal Y\cap\Gamma$ be all observed data points $\mathcal Y$ in $\Gamma$. Then the estimator $F(z)$ defined in \eqref{eq:def:F(z)} satisfies
\[
F(z) = \pi(z) + \tfrac{d}{2} H_{\pi(z)} \sigma^2 + O(\sigma^3)\qquad \text{uniformly for all } z\in \mathcal Y^o,
\]
with probability at least $1 - C_1' \exp(-C_2' \sigma^{-c'})$  
for some constants $c', C_1', C_2'>0$.
\end{cor}

\begin{cor}
Under the assumptions of Theorem~\ref{thm:main},
we obtain that
\[
\lim_{\sigma\to0^{+}} \frac{F(\pi(z)) - \pi(z)}{\sigma^2} =\tfrac{d}{2} H_{\pi(z)}.
\]
\end{cor}
\begin{rem}
    (1) The operator $F-id$ on the left-hand side can be regarded as a normalized graph Laplacian constructed with truncated kernel weights of bandwidth $r=c_D\sigma$. The limiting behavior indicates that the normalized graph Laplacian, when applied to the coordinate functions, converges to the mean curvature vector field on the manifold. This result overlaps with the theory of diffusion maps \cite{COIFMAN20065}.

    (2) The estimator $F$ on $\mathcal{M}$ acts as a discrete  mean curvature flow for a time scale of $\sigma^2$, which inherently enhances the geometric and topological structure of the data.
   
\end{rem}

\subsection{Output manifold}
We now show that small perturbations of a smooth manifold preserve its geometric regularity, in the sense that the reach of the perturbed manifold remains close to that of the original manifold.

\begin{lem}\label{lem:reach of perturbed manifold}
Let $\mathcal{M} \subset \mathbb{R}^n$ be a compact $C^3$ $d$-dimensional submanifold with reach $\tau > 0$, and $f: \mathcal{M} \to \mathbb{R}^n$ be a $C^1$ normal vector field
\[
f(x) = \frac{d}{2} H_x \sigma^2 +O(\sigma^3),
\]
where $H_x$ is the mean curvature vector of $\mathcal{M}$ at $x$.
Then for the perturbed manifold
\[
\widehat{\mathcal{M}} = \{ x + f(x) \mid x \in \mathcal{M} \},
\]
there exists a constant $c > 0$ depending only on $\mathcal{M}$ such that for sufficiently small $\sigma$,
\[
\text{reach}(\widehat{\mathcal{M}}) \geq \tau(1 - c\sigma^2)>0.
\]
\end{lem}

Next, we can demonstrate that the estimator $F$ constructed based on observation points can output an estimated manifold $F(\Gamma)$ that satisfies regularity and is sufficiently close to the original manifold. 

\begin{thm}
\label{thm:sample-level output manifold}
Let $F:\Gamma \to \mathbb{R}^D$ be the estimator as defined in \eqref{eq:def:F(z)}.  
If the sample size satisfies
$N \ge O(\sigma^{-3d-5})$,
then  
$F(\Gamma)$ is a smooth $d$-dimensional manifold with positive reach and satisfies
\[
\operatorname{dist}_H(F(\Gamma),\mathcal{M}) \le C'\sigma^2
\]
with probability at least $1 - C'_1 \exp(-C'_2 \sigma^{-c'})$ for constants  $C',C'_1,C'_2,c'>0$. 
\end{thm}
The proof uses a standard finite-covering argument over \(\Gamma\) to state the geometric properties of estimated manifold $F(\Gamma)$ with high probabilities.

\section{Numerical experiments}
This section evaluates the performance of the discrete estimator \( F \) for manifold fitting on simulated data, following an experimental setup similar to \cite{mohammed2017manifold,yao2023manifold}. Data generation and the evaluation of the reconstructed manifold proceed as follows.

First, an i.i.d. sample $\{x_i\}_{i=1}^{N}$ is generated uniformly from the $d$-dimensional submanifold $\mathcal{M}\subset \mathbb{R}^D$ with respect to the induced volume measure on $\mathcal{M}$. Independent Gaussian noise is then added to produce the observed data $\{y_i\}_{i=1}^N \subset \mathbb{R}^D$ as
 \[y_i = x_i + \xi_i, \qquad \xi_i \sim \mathcal{N}(0,\sigma^2 I_D).\] 
We then independently generate a set of test points $\{z_i\}_{i=1}^{N_0}$ within a fixed tubular neighborhood of $\mathcal{M}$, such that $\sigma/2 \leq d(z_i, \mathcal{M}) \leq 2\sigma$ (these points could include the observed sample points and are used for evaluating the algorithm).
The estimation of these test points by the estimator \( F \) is performed via the following Algorithm 1.
\begin{algorithm}[ht]\label{algorithm:1}
{\color{black}
\caption{yl25: Manifold fitting estimator.}
\raggedright \textbf{Input:} Test points $\{z_i\}_{i=1}^{N_0}$, observed samples $\{y_i\}_{i=1}^{N}$, ambient dimension $D$, noise bandwidth $\sigma$.\\
\raggedright \textbf{Output:} Estimated points $\{F(z_i)\}_{i=1}^{N_0}$ of manifold $\mathcal{M}$.\\
\begin{itemize}
    \item[1.] Calculate the truncation radius $r=\frac{D\gamma(D/2, 1)}{2\gamma(D/2 + 1, 1)}\sigma$.
    \item[2.] For each $z\in \{z_i\}_{i=1}^{N_0}$:
    \begin{itemize}
    \item Determine the index set $I_z = \bigl\{j : d(z, y_j) \leq \sqrt{2}\,r \bigr\}$.
    \item Compute the Gaussian weights $w_j=e^{\frac{-\|y_j-z\|^2}{2r^2}}$ for all $j \in I_z$.
    \item Obtain the estimated point $F(z)=\frac{\sum_{j\in I_z}^{}w_jy_j}{\sum_{j\in I_z}^{}w_j}$.
\end{itemize}
\end{itemize}
}
\end{algorithm}

Algorithm 1 requires the noise bandwidth \(\sigma\) as an input. In Sections 6.1–6.3, we use the ground-truth bandwidth from the model in order to verify the asymptotic behaviors predicted by Theorem \ref{thm:main}. In practice, \(\sigma\) can be estimated from the noisy observations \cite{li2023point,liu2006noise}.
The computational complexity of Algorithm 1 is \(O(N_0 N D)\) in time and \(O((N_0 + N) D)\) in space.
Performance evaluation is based on the supremum approximation error, defined as $\max_{j}d(F(z_j), \mathcal{M}),$ to estimate the Hausdorff distance to the true manifold $\mathcal{M}$, as well as the computational CPU time.
\subsection{Asymptotic analysis.}
In the following we employ the Hausdorff distance as an indicator to verify the asymptotic relationship established in Theorem \ref{thm:main} with respect to the noise bandwidth $\sigma$, sample size $N$, and mean curvature $H$.

\begin{figure}[ht]
    \centering   
    \includegraphics[width = 1\textwidth]{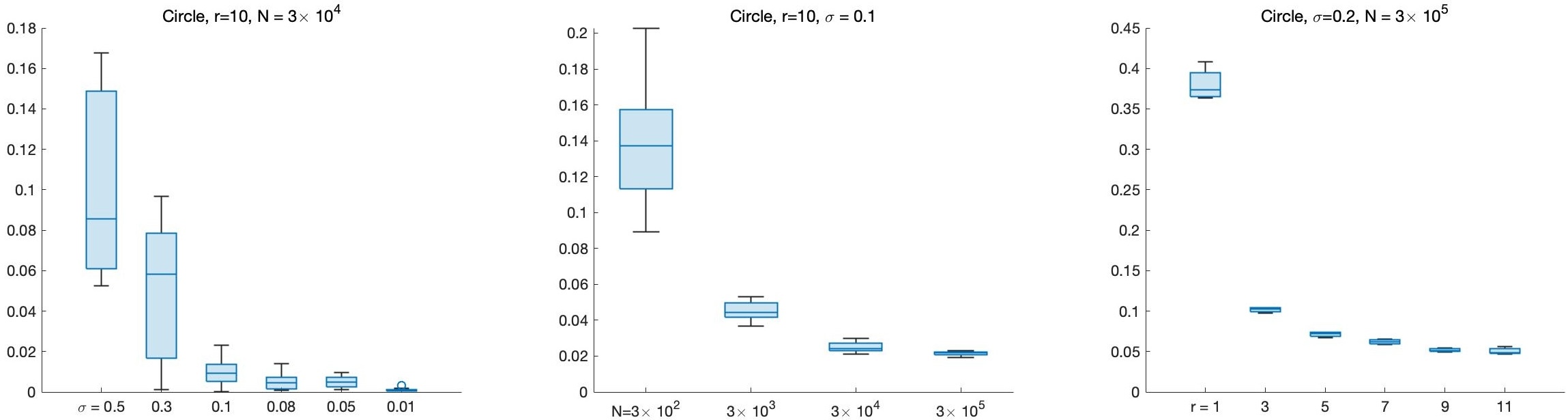}
    \caption{The asymptotic performance of yl25 with respect to the noise bandwidth $\sigma$, sample size $N$, and mean curvature $H$.}\label{Fig:asym_circle}
\end{figure}

For the case of a circle with radius $r$, we fixed \(N = 3 \times 10^4\), \(N_0 = 100\) and \(r = 10\), while varying the noise level \(\sigma \in \{0.5, 0.3, 0.1, 0.08, 0.05, 0.01\}\). For each \(\sigma\), we randomly generated 50 different sets \(\{z_i\}_{i=1}^{N_0}\) and applied the algorithm yl25 to each. The Hausdorff distances between the reconstructed manifold and the underlying manifold are displayed on the left panel of Figure \ref{Fig:asym_circle}. The results indicate that the Hausdorff distance decays quadratically as \(\sigma\) decreases, consistent with the asymptotic behavior stated in Theorem \ref{thm:main}. 
 Next, fixing \(\sigma = 0.1\) and \(r = 10\), we evaluated the performance of yl25 in different sample sizes \(N \in \{3 \times 10^2, 3 \times 10^3, 3 \times 10^4, 3 \times 10^5\}\). The corresponding Hausdorff distances are shown in the middle panel of Figure \ref{Fig:asym_circle}. As \(N\) increases, a pronounced reduction in the Hausdorff distance is observed, which can be attributed to more accurate estimations of the local geometric structure.
 Finally, setting \(N = 3 \times 10^5\) and \(\sigma = 0.2\), we examined how the Hausdorff distance varies with the mean curvature \(\kappa = r^{-1} \in \{1^{-1}, 3^{-1}, 5^{-1}, 7^{-1}, 9^{-1}, 11^{-1}\}\). The results presented in the right panel of Figure \ref{Fig:asym_circle} demonstrate that the Hausdorff distance decreases slightly as the curvature diminishes.  


\subsection{Comparison of other manifold fitting methods}
For the circle with radius $r$, we set \(N = 3000\), \(N_0 = 300\), \(r = 5\), and $\sigma = 0.1$. The parameters for each estimator were selected as follows: for ysl23 \cite{yao2023manifold}, we set \(r_0 = r_1 = 5\sigma/\lg(N)\) and \(r_2 =10\sigma\sqrt{\log(1/\sigma)}/\lg(N)\); for yx19 \cite{yao2019manifold}, cf18 \cite{fefferman2018fitting}, and km17 \cite{mohammed2017manifold}, the parameter was chosen as \(r = 2\sqrt{\sigma}\) in accordance with their papers.  
This highlights that the selection of local neighborhood parameters in these methods often requires empirical tuning, whereas our approach is guided by theoretical principles for choosing the truncation radius.

\begin{figure}[ht]
    \centering
    \includegraphics[width = 1\linewidth]{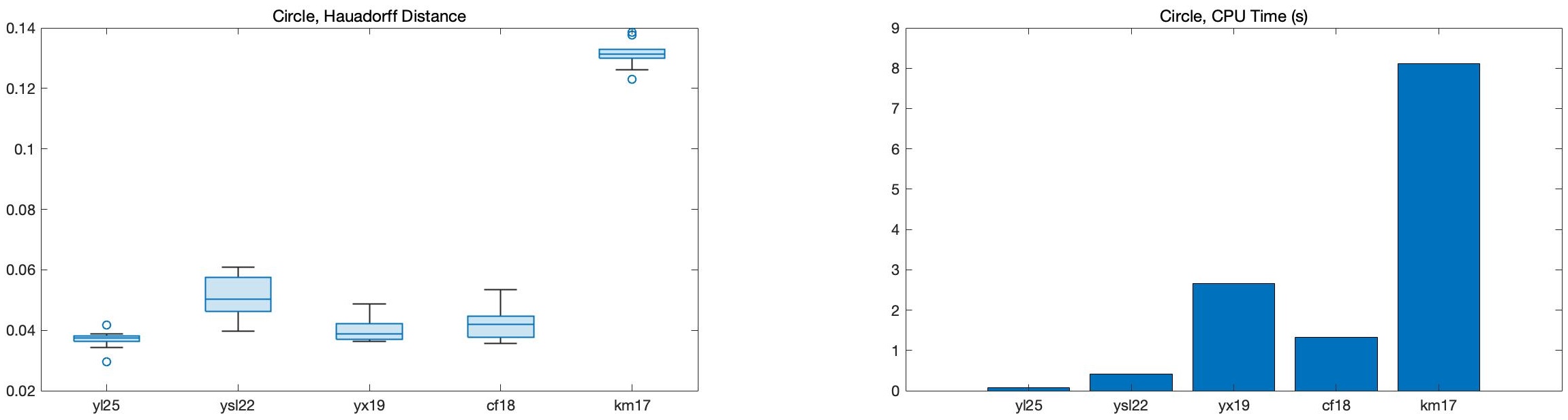}        
    \caption{The Hausdorff distance and CPU time of fitting a circle ($N = 3000$, $r = 5$, $\sigma = 0.1$) using yl25, ysl23, yx19, cf18, and km17. }\label{Fig:box_circle_sphere}
\end{figure}

Each method was executed 10 independent trials, and their performance is summarized in Figure \ref{Fig:box_circle_sphere}. In terms of Hausdorff distance, yl25 and yx19 exhibit slightly superior performance compared to ysl23 and cf18, while all four methods significantly outperform km17. Regarding computational efficiency, yl25 also demonstrates advantage, with substantially lower running times than the other four methods.

\subsection{Fitting of other manifolds}
We present the fitting performance on other non-trivial manifolds. 
For the torus, we set $N=2.5\times10^4$, $N_0=100$ and execute $10$ independent runs. 
The left and middle panels in Figure \ref{Fig:box_com_torus} compare the observed point cloud with the  fitting points obtained by our method yl25, while the right panel depicts the monotonic decay of the Hausdorff distance as the noise level decreases.

\begin{figure}[ht]
    \centering
    \includegraphics[width = 1\linewidth]{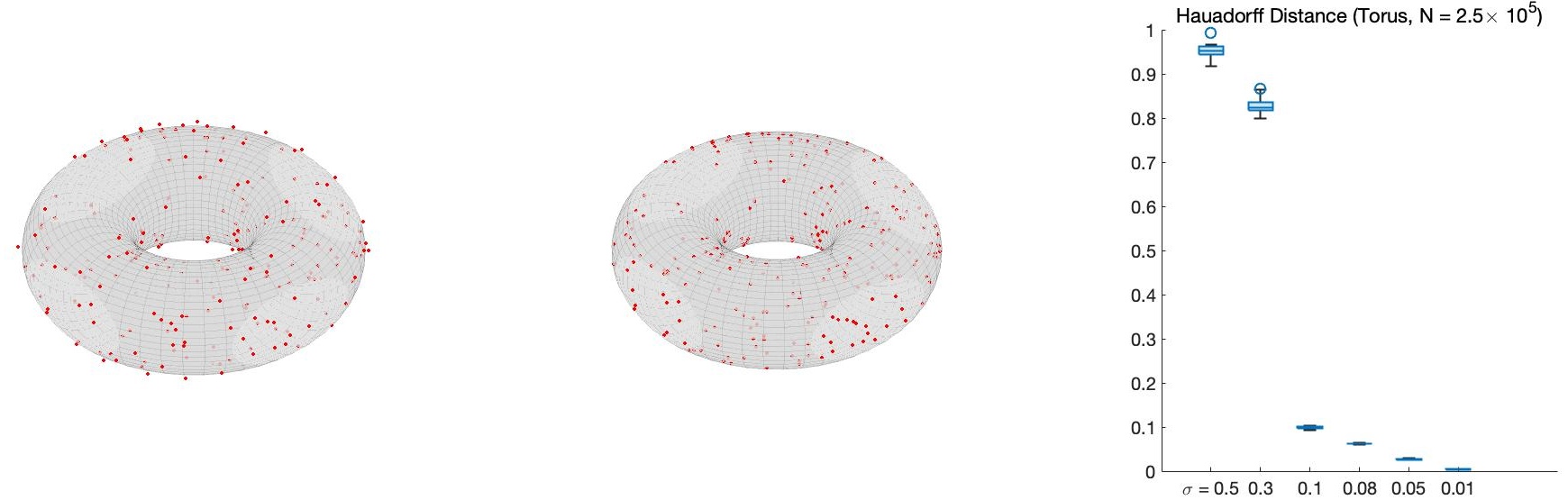}
    \caption{The fitting performance of yl25 on a torus ($N=2.5\times10^4$).}\label{Fig:box_com_torus}
\end{figure}

Analogous results and trends were also observed for the Calabi–Yau manifolds—compact Ricci-flat Kähler manifolds that are of particular interest in theoretical physics \cite{calabi2015kahler}. We consider the Fermat quartic as a concrete example, which reduces to the normalized equation \(x^4 + y^4 = 1, x, y \in \mathbb{C}\) in local coordinates. The resulting 4D surface is visualized in 3D space through dimensionality reduction and a suitable projection. Its projected morphology displayed in Figure \ref{Fig:box_com_CY}.

\begin{figure}[ht]
    \centering
    \includegraphics[width = 1\linewidth]{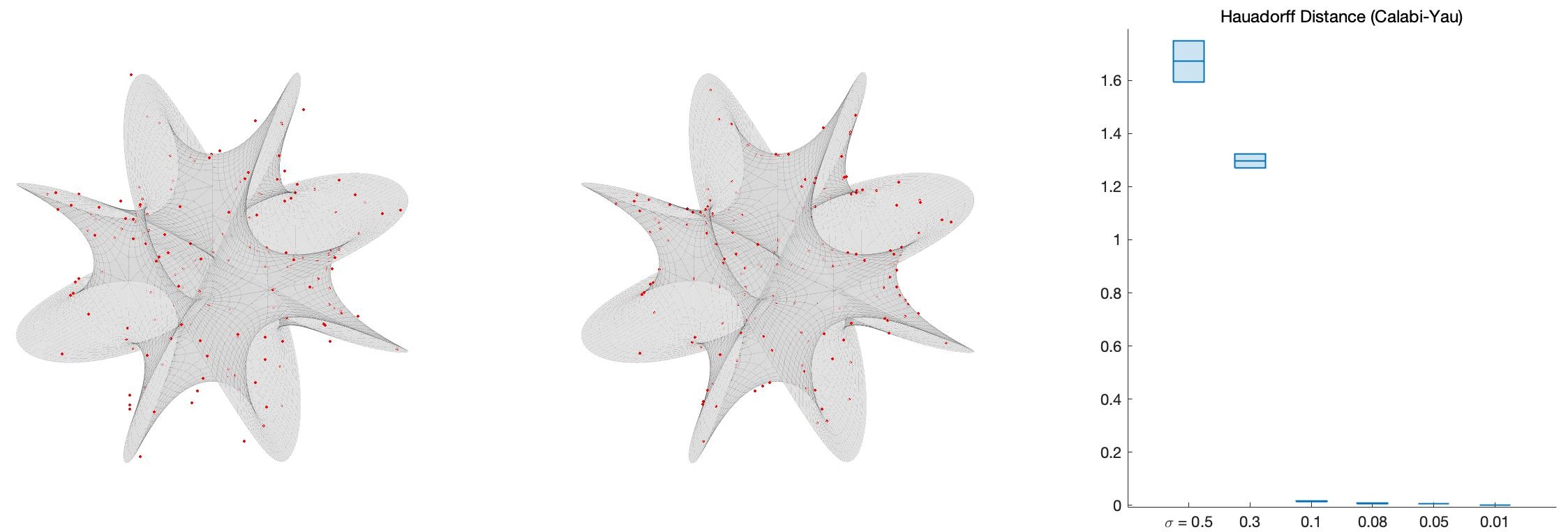}
    \caption{The fitting performance of yl25 on a Calabi–Yau manifold ($N=3\times10^5$).}\label{Fig:box_com_CY}
\end{figure}

Data points were sampled uniformly and generating \(3\times10^5\) samples in \(\mathbb{R}^4\) with added Gaussian noise. As shown in the left panel of Figure \ref{Fig:box_com_CY}, the initial noisy  lies far from the Calabi–Yau manifold, whereas the output of algorithm yl25 produces a substantially closer reconstruction (middle panel). This demonstrates the method's efficacy on complex geometric structures. Furthermore, experiments confirm that the Hausdorff distance decays as the noise level \(\sigma\) decreases. 

Numerical experiments demonstrate that our algorithm provides theoretical guidance for selecting the radius of local averaging in mean shift estimator, leading to significant improvements in both fitting accuracy and computational efficiency. However, the theoretical sample size requirement $N=O(\sigma^{-3d-5})$ increases substantially with the dimension of the underlying manifold $d$, which poses a challenge on high-dimensional latent manifolds.

\section{Conclusion}

This paper introduces an efficient method for fitting a low-dimensional manifold from noisy high-dimensional samples. Under the Gaussian noise addition model, we propose a local mean-shift estimator based on kernel-weighted averaging with a theoretically derived bandwidth \( r = c_D \sigma \). The estimator projects points in a tubular neighborhood of the underlying manifold toward it, achieving a uniform Hausdorff error of order \( O(\sigma^2) \) with high probability. It improves earlier methods relying on tangent space approximations.

Theoretical analysis shows that the leading bias is proportional to the mean curvature vector, a second-order extrinsic geometric quantity, while tangential errors are of higher order \(O(\sigma^3)\). This geometric insight not only refines the error characterization but also reveals a connection to discrete mean curvature flow
, bridging manifold learning with geometric analysis. Practically, the method requires only the noise level \(\sigma\) and avoids iterative schemes or implicit representations, offering clear input–output correspondence and computational efficiency.

Numerical experiments on circles, tori, and Calabi–Yau manifolds confirm the theoretical results that the Hausdorff distance decays quadratically with \(\sigma\). The estimator outperforms existing methods in both accuracy and runtime, and the proposed choice of bandwidth \(r = c_D \sigma\) eliminates heuristic tuning and enhances reproducibility.
However, sample complexity \( N = O(\sigma^{-3d-5}) \) grows rapidly with intrinsic dimension \(d\), posing challenges in very high-dimensional latent manifolds. 

Future work could consider more general distributional models, including non-uniform sampling on the underlying manifold and general unbounded noise. Another promising direction is the study of the discrete mean curvature flow on point cloud models and the analysis of its convergence properties.

In summary, this work provides a principled, geometrically interpretable, and practically effective framework for manifold fitting, with theoretical guarantees that align closely with empirical performance. It advances the state-of-the-art in statistical manifold estimation and offers a foundation for further research in geometric data analysis.

\section{Proofs}

\subsection{Proofs of lemmas}

\begin{proof}[Proof of Lemma ~\ref{lem:exponential map expansion}]
     Recall that $\text{exp}_{x}(u) = \gamma_u(1)$, where $\gamma_u: [0,1] \to \mathcal{M}$ is the  geodesic in $\mathcal{M}$ with initial condition $\gamma_u(0) = x$, $\gamma_u'(0) = u$. Consider the Taylor expansion of $\text{exp}_x$ around the origin of $T_x\mathcal{M}$, it will be written as
    \[
    \text{exp}_x(u) = x + I_x(u) + \frac{1}{2} Q_x(u) + \dots,
    \]
where $I_x$, $Q_x$ are respectively linear and quadratic forms in $T_x\mathcal{M}$. 
Therefore 
\[
\gamma_u(t) = x + I_x(u)t + \frac{1}{2} Q_x(u)t^2 +O(t^3). 
\]
Hence, $I_x(u) = \gamma_u'(0) = u$ and $ Q_x(u) = \gamma_u''(0)$. Since $\gamma$ is a geodesic, we have $\nabla^\mathcal{M}_{\gamma'(0)} \gamma'(0)=0$, which implies $\gamma_u''(0) = \Pi_x(u,u)$. This finishes the proof. 
\end{proof}

\begin{proof}[Proof of Lemma ~\ref{lem:volume element expansion}]
     Since $g_{ij} | _x  = \delta_{ij}$ follows that $D\text{exp}_x$ is the identity map at the origin of $T_x\mathcal{M}$, which implies $\frac{\partial}{\partial{u_i}} = d(\text{exp}_x)_0(u_i) = u_i.$ The geodesic parametrization in the exponential coordinates is given by 
\[
c(t) = t(u^1,u^2,\dots,u^d).
\]
Putting this into the geodesic equation, we get for any $k$, 
\[
0 = \ddot{c^k(t)} + \dot{c^i(t)}\dot{c^j(t)}\Gamma_{ij}^k(\text{exp}(c(t))) = u^iu^j\Gamma_{ij}^k(\text{exp}(c(t))).
\]
Let $t=0$, we have the quadratic form
\[
u^iu^j\Gamma_{ij}^k(x) = 0
\]
for any $u^iu^j$ and symmetric $\Gamma_{ij}^k$. Therefore $\Gamma^k_{ij}(0) = 0$ for all $i,j,k$, which implies $\partial_i g_{jk} + \partial_j g_{ik} - \partial_k g_{ij} = 0 $ at the origin. So $\partial_k g_{ij}|_x = 0$ follows from 
\[
\partial_k g_{ij}|_x = \frac{1}{2}(\partial_k g_{ij} + \partial_j g_{ik} - \partial_i g_{jk}+ \partial_k g_{ij} + \partial_i g_{jk} - \partial_j g_{ik})|_x = 0.
\]
Hence $\sqrt{\text{det} g_{ij}(x)} = 1$ and its first derivative vanish. Taking the Taylor expansion at $x$ we get 
\[
\sqrt{\text{det} g_{ij}(u)} = 1 + O(\|u\|^2).
\]
This finishes the proof.
\end{proof}

\begin{proof}[Proof of Lemma \ref{lem:nonasymp_hoeffding}]
Write
\[
\widehat\mu_n-\widehat\mu_w
= \frac{S_n}{B_n}-\frac{s}{b}
= \frac{S_n-s}{B_n} - \frac{s}{b}\cdot\frac{B_n-b}{B_n}.
\]
Fix the event
\[
\mathcal{E}:=\{|B_n-b|\le b/2\}.
\]
On $\mathcal{E}$ we have $B_n\ge b/2$, hence the elementary bound
\begin{equation}\label{eq:deterministic_bound}
\|\widehat\mu_n-\widehat\mu_w\|
\le \frac{2\|S_n-s\|}{b} + \frac{2\|s\|}{b^2}\,|B_n-b|.
\end{equation}
To enforce $\|\widehat\mu_n-\widehat\mu_w\|\le\varepsilon$ on $\mathcal{E}$ it suffices to require the two deterministic inequalities
\[
\|S_n-s\| \le \frac{b\varepsilon}{4},\qquad
|B_n-b| \le \frac{b^2\varepsilon}{4\|s\|}.
\]
Therefore, by the union bound,
\begin{align}
\bP\big(\|\widehat\mu_n-\widehat\mu_w\| > \varepsilon\big)
&\le \bP(\mathcal{E}^c)
+ \bP\Big(\|S_n-s\| > \frac{b\varepsilon}{4}\Big)
+ \bP\Big(|B_n-b| > \frac{b^2\varepsilon}{4\|s\|}\Big). \label{eq:union_decomp}
\end{align}

We bound the three terms on the right-hand side separately using Hoeffding's inequality (Lemma~\ref{lem:Hoeffding}).

\medskip\noindent\textbf{(i) Bound for }\(\bP(\mathcal{E}^c)=\bP(|B_n-b|>b/2)\).  
Since $0\le W\le M$, apply Hoeffding to the scalar average $B_n=\tfrac1n\sum W_i$ with range length $M$ to obtain
\begin{equation}\label{eq:Bn_tail}
\bP(|B_n-b|>b/2)\ \le\ 2\exp\!\Big(-\frac{n b^2}{2M^2}\Big).
\end{equation}
\medskip\noindent\textbf{(ii) Bound for }\(\bP\big(\|S_n-s\| > b\varepsilon/4\big)\).  
By assumption each coordinate $(W(Y)Y)^{(j)}$ is almost surely within a fixed interval of length at most $2K$ around its mean. Thus, for every coordinate $j$ Hoeffding yields
\[
\bP\Big(\big|(S_n-s)^{(j)}\big| > t\Big)
\le 2\exp\!\Big(-\frac{2n t^2}{(2K)^2}\Big) = 2\exp\!\Big(-\frac{n t^2}{2K^2}\Big).
\]
Applying the union bound over the $D$ coordinates and using $\|v\|\le\sqrt{D}\|v\|_\infty:=\sqrt{D}\max_j|v^{(j)}|$ we obtain, for any $\tau>0$,
\begin{equation*}\label{eq:Sn_tail_general}
\bP\big(\|S_n-s\| > \tau\big)
\le 2D\exp\!\Big(-\frac{n \tau^2}{2D K^2}\Big).
\end{equation*}
Putting $\tau=b\varepsilon/4$ yields
\begin{equation}\label{eq:Sn_tail}
\bP\Big(\|S_n-s\| > \frac{b\varepsilon}{4}\Big)
\le 2D\exp\!\Big(-\frac{n b^2\varepsilon^2}{32 D K^2}\Big).
\end{equation}

\medskip\noindent\textbf{(iii) Bound for }\(\bP\big(|B_n-b| > b^2\varepsilon/(4\|s\|)\big)\).  
Again applying Hoeffding to $B_n$ with threshold $b^2\varepsilon/(4\|s\|)$ yields
\begin{equation}\label{eq:Bn_tail_second}
\bP\Big(|B_n-b| > \frac{b^2\varepsilon}{4\|s\|}\Big)
\le 2\exp\!\Big(-\frac{2n s^2}{M^2}\Big)
= 2\exp\!\Big(-\frac{n b^4\varepsilon^2}{8\|s\|^2 M^2}\Big).
\end{equation}

Combining the three bounds \eqref{eq:Bn_tail}, \eqref{eq:Sn_tail}, \eqref{eq:Bn_tail_second} with \eqref{eq:union_decomp} we obtain the explicit finite-sample tail bound
\begin{equation*}
\bP\big(\|\widehat\mu_n-\widehat\mu_w\| > \varepsilon\big)
\le 2\exp\!\Big(-\frac{n b^2}{2M^2}\Big)
+ 2D\exp\!\Big(-\frac{n b^2\varepsilon^2}{32 D K^2}\Big)
+ 2\exp\!\Big(-\frac{n b^4\varepsilon^2}{8\|s\|^2 M^2}\Big).
\end{equation*}
This completes the proof.
\end{proof}

\subsection{Proofs of Section 3}

\begin{proof}[Proof of Theorem \ref{thm: loc average}]
We expand $p_\sigma(y)$ in a Taylor series around $z$ as
\begin{align*}
     p_\sigma(y)
    &= p_\sigma(z)
    + \sum_{k=1}^D (y^{(k)} - z^{(k)})\,\partial_{k} p_\sigma(z)\\
    &+ \frac{1}{2}\sum_{k,l=1}^D (y^{(k)} - z^{(k)})(y^{(l)} - z^{(l)})\,\partial_{kl}^2 p_\sigma(z)
    + O(\|y-z\|^3).
\end{align*}
Since both the kernel $\varphi_r(y-z)$ and the integration domain $B_D(z,\sqrt{2}r)$ are radially symmetric, all terms that are odd in $(y-z)$ vanish upon integration. Hence,
\begin{align}
   \label{eq:denominator} &\int_{B_D(z,\sqrt{2}r)} \varphi_r(y-z)\,p_\sigma(y)\,dy \nonumber\\
    &= \int_{B_D(z,\sqrt{2}r)} \varphi_r(y-z)
      \Big(p_\sigma(z) + \tfrac{1}{2}|y^{(1)}-z^{(1)}|^2\,\Delta p_\sigma(z) + O(\|y-z\|^4)\Big) dy.
\end{align}
Letting $y = z + \sqrt{2}r\,u$ and using $dy = ( \sqrt{2}r )^D du$, we obtain
\begin{align*}
    &2^{-\frac{D}{2}}\int_{B_D(z,\sqrt{2}r)} \varphi_r(y-z)\,p_\sigma(y)\,dy\\
    &= \!\int_{B_D(0,1)} \varphi_1(\sqrt{2}u)\,
       \Big(p_\sigma(z) + r^2 (u^{(1)})^2 \Delta p_\sigma(z) + O(r^4)\Big) du \\
    &= Ap_\sigma(z) + Br^2 \Delta p_\sigma(z) + O(r^4),
\end{align*}
where
\[
    A = \int_{B_D(0,1)} \varphi_1(\sqrt{2}u)\,du, \qquad
    B = \!\int_{B_D(0,1)}  (u^{(1)})^2\,\varphi_1(\sqrt{2}u)\, du.
\]
Similarly, since $\Delta(y\,p_\sigma(y)) = 2\nabla p_\sigma(y) + y\,\Delta p_\sigma(y)$, we have
\begin{equation}\label{eq:numerator}
\begin{aligned}
    &2^{-\frac{D}{2}}\int_{B_D(z,\sqrt{2}r)} \varphi_r(y-z)\,y\,p_\sigma(y)\,dy\\
    &= Az\,p_\sigma(z)
      + Br^2\big(2\nabla p_\sigma(z) + z\,\Delta p_\sigma(z)\big)
      + O(r^4).
\end{aligned}
\end{equation}
Then \eqref{eq:denominator} and \eqref{eq:numerator} yields
\begin{align*}
    \mu_z 
    &= \frac{Azp_\sigma(z) + Br^2(2\nabla p_\sigma(z) + z\Delta p_\sigma(z)) + O(r^4)}
           {Ap_\sigma(z) + Br^2\Delta p_\sigma(z) + O(r^4)} \\
    &= z + r^2\,\frac{2B\,\nabla p_\sigma(z)}{A\,p_\sigma(z)} + O(r^4),
\end{align*}
which concludes the proof.
\end{proof}

\begin{proof}[Proof of Theorem \ref{th:probability noisy}]
Fix any $\varepsilon\in(0,\tau)$ and $y\in T(\tau-\varepsilon)$. We assume $0<\sigma<\sigma_0$ for some $\sigma_0>0$. Let $x_0=\pi(y)\in \mathcal{M}$, $v:=y-x_0\in T^\bot_{x_0}\mathcal{M}$ so that $\|v\|\le \tau-\varepsilon$. Set
\[
C_\sigma:=\big(\operatorname{Vol}(\mathcal{M})(2\pi\sigma^2)^{D/2}\big)^{-1},\quad 
p_\sigma(y)=C_\sigma\int_\mathcal{M} \exp\!\Big(-\tfrac{1}{2\sigma^2}\|y-x\|^2\Big)\,d\mu(x).
\]
Define the symmetric endomorphism $\langle v,\Pi_{x_0}\rangle$ on $T_{x_0}\mathcal{M}$ via
$\langle \langle v,\Pi_{x_0}\rangle u,w\rangle=\langle v,\Pi_{x_0}(u,w)\rangle$, and set
\(
A:=I-\langle v,\Pi_{x_0}\rangle.
\)
Using $\|\Pi_{x_0}\|_{\mathrm{op}}\le \tau^{-1}$ (Lemma~\ref{lem:second fundamental form bound}) and $\|v\|\le \tau-\varepsilon$, then
\[
\lambda_{\min}(A)\ \ge\ 1-\|v\|/\tau\ \ge\ \varepsilon/\tau\ =:\ c_\varepsilon>0.
\]
Hence, $A$ is uniformly positive definite on $T(\tau-\varepsilon)$ with a spectral gap $c_\varepsilon$. 

In the exponential coordinates at $x_0$, we write $x(u):=\exp_{x_0}(u)$ for $u\in T_{x_0}\mathcal{M}\simeq\mathbb{R}^d$. By Lemma~\ref{lem:injectivity radius}, there exists $r_0>0$ such that $u\mapsto x(u)$ is a diffeomorphism from $B_d(0,r_0)$ onto $E_0:=\exp_{x_0}(B_d(0,r_0))$. We have the following expansion on $B_d(0,r_0)$,
\[
\sqrt{\det g(u)}=1+Q_g(u),\qquad |Q_g(u)|\le C_g\,\|u\|^2, \quad(\text{Lemma}~\ref{lem:volume element expansion})
\]
\[
x(u)=x_0+u+\tfrac12\,\Pi_{x_0}(u,u)+Q_e(u),\qquad \|Q_e(u)\|\le C_e\,\|u\|^3, \quad(\text{Lemma}~\ref{lem:exponential map expansion})
\]
with constants $C_g,C_e$ depending only on $\mathcal{M}$. By the expansion of $x(u)$ we obtain
\[
\|y-x(u)\|^2 = \|v\|^2 + u^\top A u + R(u),\qquad R(u) = O\bigl(\|v\|\,\|u\|^3 + \|u\|^4\bigr),
\]
with constants depending only on $\mathcal{M},\varepsilon$. 
For $r=r_0$ and let
\[
E := \exp_{x_0}\bigl(B_d(0,r)\bigr)\subset \mathcal{M},\qquad
P_\sigma(y) = C_\sigma\,(I_1 + I_2),
\]
where
\[
I_1 := \int_E e^{-\|y-x\|^2/(2\sigma^2)}\,d\mu(x),\qquad
I_2 := \int_{\mathcal{M}\setminus E} e^{-\|y-x\|^2/(2\sigma^2)}\,d\mu(x).
\]

Using the exponential coordinates $x=x(u)$ and the above expansions, we have
\[
I_1
=
\int_{\|u\|\le r}
\exp\Bigl(-\frac{1}{2\sigma^2}\bigl(\|v\|^2 + u^\top A u + R(u)\bigr)\Bigr)
\bigl(1+Q_g(u)\bigr)\,du.
\]
 Change variables $z = u/\sigma$, so that $du=\sigma^d dz$. Using the bounds of  $R(u)=O(\|v\|\|u\|^3+\|u\|^4)$, $Q_g(u)=O(\|u\|^2)$ and the Taylor expansion of the exponential, we obtain
\[
\begin{aligned}
I_1
&=
e^{\frac{-\|v\|^2}{2\sigma^2}}\sigma^d
\int_{\|z\|\le r/\sigma}
e^{-\frac{1}{2} z^\top A z}
\Bigl(1 - \frac{R(\sigma z)}{2\sigma^2} + O\Bigl(\frac{R(\sigma z)^2}{\sigma^4}\Bigr)\Bigr)
\bigl(1+Q_g(\sigma z)\bigr)\,dz \\
&=
e^{\frac{-\|v\|^2}{2\sigma^2}}\sigma^d
\int_{\|z\|\le r/\sigma}
e^{-\frac{1}{2} z^\top A z}
\Bigl(1 + \sigma\|v\|\phi_1(y,z,\sigma) + \sigma^2\phi_2(y,z,\sigma)\Bigr)\,dz,
\end{aligned}
\]
where $\phi_1,\phi_2$ are smooth functions of $(y,z,\sigma)$ and satisfy polynomial bounds of the form
\[
|\phi_i(y,z,\sigma)| \le C\|z\|^m
\]
for some $m$ independent of $(y,\sigma)$. Using $\mathcal{M}$ is a compact $C^3$ submanifold and the tubular neighborhood theorem, all curvature quantities and the projection $\pi(\cdot)$, as well as their derivatives, are uniformly bounded on
$T(\tau - \varepsilon)$. Thus for $|\alpha|\le1$, there exist an integer $m_0$ and constants $C_\alpha>0$ such that
\[
\bigl|\partial_y^\alpha\phi_i(y,z,\sigma)\bigr|
\;\le\;
C_\alpha\,\|z\|^{m_0}.
\]
Since $\lambda_{\min}(A)\ge c_\varepsilon$, the Gaussian tail is exponentially small:
\[
\int_{\|z\|>r/\sigma}e^{-\frac12 z^\top A z}dz
\ \le\ e^{-\frac{c_\varepsilon r^2}{4\sigma^2}}\int_{\mathbb{R}^d}e^{-\frac14 z^\top A z}dz
\ \le C e^{-c/\sigma^2}.
\]
Thus we may extend the integral to all of $\mathbb{R}^d$ with exponentially small error and write
\[
I_1
=
e^{-\|v\|^2/(2\sigma^2)}\sigma^d
\Bigl(
J_0(y) + \sigma\|v\| J_1(y,\sigma) + \sigma^2 J_2(y,\sigma)
\Bigr),
\]
where
\[
J_0(y) := \int_{\mathbb{R}^d} e^{-\frac{1}{2} z^\top A z}\,dz
=
\frac{(2\pi)^{d/2}}{\sqrt{\det A}},
\]
and
\[
J_i(y,\sigma) :=
\int_{\mathbb{R}^d} e^{-\frac{1}{2} z^\top A z}\,\phi_i(y,z,\sigma)\,dz,
\quad i=1,2.
\]
Using the boundedness of $\phi_i(y,z,\sigma)$, for the Gaussian moments in the integral we obtain
\[
\bigl|\partial_y^\alpha\bigl(e^{-\frac12 z^\top A_y z}\phi_i(y,z,\sigma)\bigr)\bigr|
\;\le\;
C_\alpha'\,e^{-\frac{c_\varepsilon}{2}\|z\|^2}\,\|z\|^{m_0},
\]
Hence by dominated convergence we obtain
\[
\sup_{(y,\sigma)\in T(\tau-\varepsilon)\times(0,\sigma_0]}
\bigl|\partial_y^\alpha J_i(y,\sigma)\bigr|
\;\le\; C_\alpha'',
\]
for some constant $C_\alpha''<\infty$. In particular, $J_i$ are smooth in $y$, and all $y$–derivatives are uniformly bounded on $T(\tau-\varepsilon)\times(0,\sigma_0]$.

For $I_2$, on the boundary $\|u\| = r$,  by $\lambda_{\min}(A)\ge c_\varepsilon$ we obtain
\[
\|y-x(u)\|^2=\|v\|^2+u^\top A u+R(u)\ \ge\ \|v\|^2+ \delta_\varepsilon,\qquad
\delta_\varepsilon:=\tfrac{c_\varepsilon}{2}\,r^2=\tfrac{\varepsilon}{2\tau}\,r^2.
\]
Then the set $\{\|y-x\|^2<\|v\|^2+\delta_\varepsilon\}$ is contained in $E$. Hence for  $x\in \mathcal{M}\setminus E$, $\|y-x\|^2\ge \|v\|^2+\delta_\varepsilon$, and 
\[
I_2\le\ \mathrm{Vol}(\mathcal{M})\,\exp\!\Big(-\frac{\|v\|^2+\delta_\varepsilon}{2\sigma^2}\Big)
\ =\ e^{-\frac{\|v\|^2}{2\sigma^2}}\;O\!\big(e^{-\delta_\varepsilon/(2\sigma^2)}\big),
\]
which is exponentially smaller than any power of $\sigma$, the term $I_2$ can be absorbed into the $\sigma^2 J_2$ contribution by adjusting $J_2$.

Finally, combine $I_1$ and $I_2$ together. Let $F_i(y,\sigma) := J_i(y,\sigma)/J_0(y)$ and using $C_\sigma\sigma^d J_0(y)
= (\operatorname{Vol}(\mathcal{M})(2\pi\sigma^2)^{\frac{D-d}{2}})^{-1} \det(A)^{-1/2}$, we obtain
\[
p_\sigma(y)
=
\frac{e^{-\|v\|^2/(2\sigma^2)}}{\operatorname{Vol}(\mathcal{M})(2\pi\sigma^2)^{\frac{D-d}{2}}}
\frac{1}{\sqrt{\det A}}
\Bigl(
1 + \sigma\|v\| F_1(y,\sigma) + \sigma^2 F_2(y,\sigma)
\Bigr),
\]
holds for all $y\in T(\tau-\varepsilon)$ and $0<\sigma\le\sigma_0$. Since $\lambda_{\min}(A)\ge c_\varepsilon$, $J_0(y)$ is bounded away from zero and smooth in $y$, then $F_1,F_2$ with all derivatives are uniformly bounded. 
\end{proof}

\begin{proof}[Proof of Corollary \ref{co:log density}]
Fix any $\varepsilon\in(0,\tau)$ and consider the set $K:=T(\tau-\varepsilon)\times(0,\sigma_0]$.
By Theorem~\ref{th:probability noisy}, we have
\begin{equation}
p_\sigma(y)
=
C_\sigma\,
\frac{e^{-\frac{\|v_y\|^2}{2\sigma^2}}}{\sqrt{\det A_y}}\,
\bigl(1+R_P(y,\sigma)\bigr),\;C_\sigma:=\big(\operatorname{Vol}(\mathcal{M})(2\pi\sigma^2)^{D/2}\big)^{-1},
\label{eq:P-factor}
\end{equation}
where
\begin{equation}
R_P(y,\sigma)
=
\sigma\|v_y\|\,G_1(y,\sigma)
+
\sigma^2 G_2(y,\sigma).
\label{eq:RP-structure}
\end{equation}
Here $G_1,G_2$ are smooth functions of $(y,\sigma)$ on $K$, arising from
Gaussian integrals of smooth kernels in the proof of Theorem~\ref{th:probability noisy}. In particular, there exist constants $C_m>0$ such that for $m= 1,2$,
\begin{equation}
\sup_{(y,\sigma)\in K}
\bigl\|\nabla_y^m G_i(y,\sigma)\bigr\|
\le C_m,
\qquad i=1,2.
\label{eq:Fi-bounded}
\end{equation}
By the tubular neighborhood theorem, the projection
$\pi:T(\tau-\varepsilon)\to \mathcal{M}$ and the normal component $v_y=y-\pi(y)$ are
$C^{2}$ in $y$, then their derivatives are uniformly bounded
on $T(\tau-\varepsilon)$. That is, for $m= 1,2$ there exists $B_m>0$ such that
\begin{equation}
\sup_{y\in T(\tau-\varepsilon)} \bigl\|\nabla_y^m v_y\bigr\|
\le B_m.
\label{eq:v-bounded}
\end{equation}
The same holds for $\|v_y\|$, since it is obtained from $v_y$ by smooth
operators.

We now bound $\nabla_y^m R_P$ using \eqref{eq:Fi-bounded} and \eqref{eq:v-bounded}. For $m=0$, we have
\[
|R_P(y,\sigma)|
\le
\sigma\|v_y\|\bigl|G_1(y,\sigma)\bigr|
+
\sigma^2\bigl|G_2(y,\sigma)\bigr|
\le
C_0\bigl(\sigma\|v_y\|+\sigma^2\bigr),
\]
for some constant $C_0$. For $m= 1$, since each derivative $\nabla_y R_P$ is a finite sum of terms of the form
\[
\sigma\,
\nabla_y^{k_1}\|v_y\|\,
\nabla_y^{k_2}G_1(y,\sigma)
\quad\text{or}\quad
\sigma^2\,\nabla_y^{k}G_2(y,\sigma),
\qquad k_1+k_2\le 1,\ k\le 1,
\]
there exist constants $C_{\varepsilon,1}>0$ depending only on $\varepsilon$ and $\mathcal{M}$ such that
\begin{equation}
\bigl\|\nabla_y R_P(y,\sigma)\bigr\|
\le
C_{\varepsilon,1}\bigl(\sigma\|v_y\|+\sigma^2\bigr).
\label{eq:RP-derivative-bound}
\end{equation}

Taking logarithms in \eqref{eq:P-factor}, we obtain
\[
\log p_\sigma(y)
=
\log C_\sigma
-
\frac{\|v_y\|^2}{2\sigma^2}
-
\frac{1}{2}\log\det A_y
+
\log\bigl(1+R_P(y,\sigma)\bigr).
\]
Let 
\[
R_0(y,\sigma) := \log\bigl(1+R_P(y,\sigma)\bigr).
\]
Shrinking $\sigma_0$ if necessary, we may assume
\[
|R_P(y,\sigma)| \le \frac{1}{2},
\qquad (y,\sigma)\in K.
\]
Then $1+R_P$ is bounded away from zero. For $m=0$, simple inequality yields
\[
|R_0(y,\sigma)|
=
\bigl|\log(1+R_P(y,\sigma))\bigr|
\le
2\,|R_P(y,\sigma)|
\le
2C_{\varepsilon,0}\bigl(\sigma\|v_y\|+\sigma^2\bigr).
\]
For $m= 1$, derivative $\nabla_y R_0(y,\sigma)$ is a finite sum of terms of the
form
\[
c_{\ell}
\partial_y^{l}R_P(y,\sigma),
\]
where $c_{\ell}$ are smooth functions involving derivatives of logarithm. 

By \eqref{eq:RP-derivative-bound} we obtain
\[
\bigl\|\nabla_y R_0(y,\sigma)\bigr\|
\le
C_{\varepsilon,1}'\bigl(\sigma\|v_y\|+\sigma^2\bigr),
\]
for some constants $C_{\varepsilon,1}'$ depending only on $\varepsilon$ and $\mathcal{M}$.
\end{proof}

\begin{proof}[Proof of Lemma \ref{lem:derivative of v pi}]
   For any $y \in \mathcal T(\tau)$, write $x := \pi(y) \in \mathcal{M}$ and $v := v_y = y - x \in
T_x^\perp \mathcal{M}$. Let $S_v : T_x\mathcal{M} \to T_x\mathcal{M}$ denote the shape operator associated
with the normal vector $v$, defined by 
\[
\langle S_v (u), w \rangle = \langle \Pi_x(u,w), v \rangle, \qquad S_v(u) = -(\nabla_u v)^\top,
\]
for $u,w\in T_x\mathcal{M}$. Then we have $A_y = I_d - \langle v_y,\Pi_{x}\rangle = I_d - S_v$ on $T_x\mathcal{M}$.

Let $\omega \in \mathbb{R}^D$ be arbitrary, and choose a smooth curve $\gamma : (-\delta,\delta) \to T(\tau)$ such that $\gamma(0) = y$ and $\gamma'(0) = \omega$. Set
\[
\pi(t) := \pi(\gamma(t)) \in \mathcal{M}, \qquad v(t) := v_{\gamma(t)} = \gamma(t) - \pi(t)
\in T_{\pi(t)}^\perp \mathcal{M}.
\]
Then for any tangent vector field $Y(t) \in T_{\pi(t)}\mathcal{M}$ along $\pi(t)$ we have
\[
\langle v(t), Y(t) \rangle = 0, \qquad \forall t.
\]
Differentiating this identity with respect to $t$ and evaluating at $t=0$, we obtain
\begin{equation}\label{eq:parametrized_derivative} 
0 = \frac{d}{dt}\langle v(t), Y(t)\rangle\big|_{t=0}
= \langle v'(0), Y(0)\rangle + \langle v(0), \nabla_{\pi'(0)} Y\rangle,
\end{equation}
where $v'(0) = \nabla_\omega v$ and $\pi'(0) = \nabla_\omega \pi(y) \in T_{\pi(y)}\mathcal{M}$.
Using the definition of the shape operator (and extending $v$ locally if
necessary), we have
\[
\langle v(0), \nabla_{\pi'(0)} Y \rangle=-\langle (\nabla_{\pi'(0)}v)^\top, Y(0) \rangle
= - \langle S_v(\pi'(0)), Y(0) \rangle.
\]
Substituting this into (\ref{eq:parametrized_derivative}) gives
\[
\langle v'(0), Y(0)\rangle = - \langle S_{v}(\pi'(0)), Y(0)\rangle,
\]
which implies that the tangential component of $\nabla_\omega v$ satisfies
\begin{equation}\label{eq:tangent_derivative}
(\nabla_\omega v)^\top = - S_v(\nabla_\omega \pi(y)).
\end{equation}
On the other hand, by the definition $v(t) = \gamma(t) - \pi(t)$, 
\begin{equation}\label{eq:basic_derivative}
\nabla_\omega v = \nabla_\omega(\gamma - \pi) = \omega - \nabla_\omega \pi(y).
\end{equation}
Taking the tangential component of \eqref{eq:basic_derivative} and combining it with \eqref{eq:tangent_derivative} yields
\[
\omega^\top - \nabla_\omega \pi(y) = - S_v(\nabla_\omega \pi(y)).
\]
Hence,
\begin{equation}\label{eq:main_relation}
\omega^\top = (I_d - S_v)\big(\nabla_\omega \pi(y)\big).
\end{equation}
By Lemma~\ref{lem:second fundamental form bound} and the reach assumption, we have $\|S_v\|_{\mathrm{op}} \le \| \Pi_x\|_{\mathrm{op}} \|v\| \le \tau^{-1}\|v\| < 1$ for $y \in \mathcal T(\tau)$, hence $A_y = I - S_v$ is invertible on $T_x\mathcal{M}$. Therefore
\[
\nabla_\omega \pi(y) = A_y^{-1} \omega^\top.
\]
Finally, substituting this back into \eqref{eq:basic_derivative}, we obtain
\[
\nabla_\omega v_y = \omega - A_y^{-1}\omega^\top.
\]
\end{proof}

\begin{proof}[Proof of Theorem \ref{th:score_noisy}]
For any $\varepsilon \in (0,\tau)$, let $\sigma_0 > 0$ be as in Corollary~\ref{co:log density}.
For $y \in T(\tau-\varepsilon)$ and $0 < \sigma \le \sigma_0$, Corollary~\ref{co:log density} gives the expansion
\begin{equation}\label{eq:logP-expansion}
\log p_\sigma(y)
= C_\sigma - \frac{\|v_y\|^2}{2\sigma^2}
  - \frac{1}{2}\log\det A_y + R(y,\sigma),
\end{equation}
where $C_\sigma = \log\big( \operatorname{Vol}(\mathcal{M}) (2\pi\sigma^2)^{(D-d)/2}\big)^{-1}$ and the remainder $R$ satisfies
\[
\|\nabla_y^m R(y,\sigma)\|_{\mathrm{op}}
\le C_{\varepsilon,m} \bigl( \sigma \|v_y\| + \sigma^2\bigr),\quad m = 0,1.
\]
Taking the ambient gradient of \eqref{eq:logP-expansion} yields
\begin{equation}\label{eq:score-decomposition}
\nabla \log p_\sigma(y)
= - \frac{1}{2\sigma^2}\nabla\|v_y\|^2
  - \frac{1}{2}\nabla \log\det A_y + \nabla R(y,\sigma).
\end{equation}
Using Lemma~\ref{lem:derivative of v pi}, for any $\omega \in \mathbb{R}^D$ we have
\[
\nabla_\omega v_y = \omega - A_y^{-1}\omega^\top.
\]
Since $v_y \in T_{\pi(y)}^\perp \mathcal{M}$ and $A_y^{-1}\omega^\top \in T_{\pi(y)}\mathcal{M}$, we obtain
\[
D_\omega \|v_y\|^2
= 2\langle v_y, \nabla_\omega v_y \rangle
= 2\langle v_y, \omega \rangle.
\]
By the definition of the gradient, this implies
\[
\nabla \|v_y\|^2 = 2 v_y.
\]
Consequently, the first term in \eqref{eq:score-decomposition} is exactly
\[
- \frac{1}{2\sigma^2}\nabla\|v_y\|^2 = - \frac{v_y}{\sigma^2}.
\]
Define
\[
G(y) := -\frac{1}{2}\log\det A_y, \qquad y \in T(\tau-\varepsilon).
\]
For $x \in \mathcal{M}$ we have $v_x = 0$ and $A_x = I_d$, so $\det A_x = 1$ and
$G(x) = 0$. Let $u \in T_x^\perp \mathcal{M}$ with $\|u\| = 1$ and consider the curve
\[
\gamma(t) := x + t u, \qquad |t| \ll 1.
\]
For $|t|$ small, $\gamma(t) \in T(\tau-\varepsilon)$, $\pi(\gamma(t)) = x$ and
$v_{\gamma(t)} = t u$. Let $S_u$ be the shape operator defined at $x$ associated with $u$. By linearity of the shape operator in the normal argument, we have
\[
A_{\gamma(t)} = I_d - \langle t u, \Pi_x \rangle
= I_d - t\, S_u,
\]
Therefore
\[
G(\gamma(t)) = -\frac{1}{2}\log\det\bigl(I_d - t S_u\bigr).
\]
Differentiating at $t=0$ gives
\[
\frac{d}{dt} G(\gamma(t))\Big|_{t=0}
= -\frac{1}{2}\,\mathrm{Tr}(-S_u)
= \frac{1}{2}\mathrm{Tr}(S_u).
\]
For mean curvature $H_x = \frac{1}{d}\mathrm{Tr}(\Pi_x)$, we have $\mathrm{Tr}(S_u) = d \langle H_x, u\rangle$. Hence
\[
\frac{d}{dt} G(\gamma(t))\Big|_{t=0}
= \frac{d}{2} \langle H_x, u\rangle,
\]
which shows that
\[
\nabla G(x) = \frac{d}{2} H_x \in T_x^\perp \mathcal{M}.
\]
Now let $y \in T(\tau-\varepsilon)$. According to Lemma~\ref{lem:Lip-derivative}, we have 
\[
\|\nabla G(y) - \nabla G(\pi(y))\| \le C \|v_{y}\|
\]
for some constant $C > 0$ depending only on $\mathcal{M}$ and $\varepsilon$. Therefore
\[
-\frac{1}{2}\nabla \log\det A_y
= \nabla G(y)
= \frac{d}{2} H_{\pi(y)} + O(\|v_y\|).
\]
According to the regularity of $A_y$, $\nabla F(y)$ is bounded uniformly on $T(\tau-\varepsilon)$.

by Corollary~\ref{co:log density}
\[
\|\nabla R(y,\sigma)\|
\le C_{\varepsilon,1}\bigl(\sigma \|v_y\| + \sigma^2\bigr).
\]
Substituting the above gradients into \eqref{eq:score-decomposition}, we obtain
\[
\nabla \log p_\sigma(y)
= - \frac{v_y}{\sigma^2}
  + \frac{d}{2}H_{\pi(y)}
  + O(\|v_y\| +\sigma \|v_y\| + \sigma^2),
\]
where the $O(\cdot)$ term is uniform in $y \in T(\tau-\varepsilon)$ and $0 < \sigma \le \sigma_0$. 
\end{proof}

\begin{proof}[Proof of Theorem \ref{thm:Population-level output manifold}]
For any fix $\varepsilon \in (0,\tau)$, let $\sigma_0 > 0$ be as in Theorem~\ref{th:score_noisy}.
For $0 < \sigma \le \sigma_0$,
by Theorem~\ref{thm: loc average} with $r^2=\tfrac{A}{2B}\sigma^2=c_D^2\sigma^2$, we may estimate the population mapping $\mu_z$
at each $z\in T(\tau-\varepsilon)$ as
\begin{align*}
\mu_z
&=\frac{\mathbb E[\varphi_r(Y-z)Y]}{\mathbb E[\varphi_r(Y-z)]}
=\frac{\int_{B_D(z,\sqrt{2}r)} \varphi_\sigma(y-z)\,y\,p_\sigma(y)\,dy}
       {\int_{B_D(z,\sqrt{2}r)} \varphi_\sigma(y-z)\,p_\sigma(y)\,dy}\\
&=z+\sigma^2 \nabla\log p_\sigma(z)+O(\sigma^4)\\
&=z-v_z+\sigma^2\frac{d}{2}H_{\pi(z)}+\sigma^2O(\|v_z\|+\sigma\|v_z\|+\sigma^2)\\
&=\pi(z)+\sigma^2\frac{d}{2}H_{\pi(z)}+O(\sigma^2).
\end{align*}

Hence $\mu(T(\tau-\varepsilon))$ is contained in a tubular neighborhood of $\mathcal M$ 
with radius $C\sigma^2$ for some constant $C$, that is, $\mu(T(\tau-\varepsilon))\subset T(C\sigma^2)$ and
\[
\sup_{z\in T(\tau-\varepsilon)} d(\mu_z,\mathcal{M}) \le C\sigma^2.
\]
In contrast, for every $x\in\mathcal{M}$ there exists $z=x\in T(\tau-\varepsilon)$ such that $\mu_z=x+O(\sigma^2)v\in T_x^\perp \mathcal{M}$, which implies that
\[
\sup_{x\in\mathcal{M}} d(x,\mu(T(\tau-\varepsilon)))\le C\sigma^2.
\]
Then $\mu(T(\tau-\varepsilon))$ is a smooth submanifold satisfying
\[
\operatorname{dist}_H(\mu(T(\tau)),\mathcal M)\le C\sigma^2.
\]
This completes the proof.
\end{proof}

\subsection{Proofs of Section 4}

\begin{proof}[Proof of Theorem \ref{thm:Concentrate_with_n}]
    For any $z\in \Gamma$, suppose $\sigma\in(0,\sigma_0]$. By Lemma \ref{lem:nonasymp_hoeffding}, let $W(y)=\varphi_r(y-z)$, then $M=O(r^{-D})=O(\sigma^{-D})$ and $K:=2rM=O(\sigma^{1-D})$. 
   Using the argument of Theorem \ref{thm: loc average}, we obtain 
    \begin{align*}
    b&=\int_{B_D(z,\sqrt{2}r)} \varphi_r(y-z)\,p_\sigma(y)\,dy\\
    &= \!\int_{B_D(0,1)} \!\varphi_1(\sqrt{2}u)\,
       \Big(p_\sigma(z) + r^2 (u^{(1)})^2 \Delta p_\sigma(z) + O(r^4)\Big) 2^{\frac{D}{2}}du \\
    &=Cp_\sigma(z)  + O(r^2),
\end{align*}
where $C =  2^{\frac{D}{2}}\int_{B_D(0,1)} \varphi_1(\sqrt{2}u)\,du$.
Similarly, 
\begin{equation}
    s=\int_{B_D(z,\sqrt{2}r)} \varphi_r(y-z)\,y\,p_\sigma(y)\,dy
    = Cz\,p_\sigma(z)+ O(r^2).
\end{equation}
 Recall that in Theorem \ref{th:probability noisy} we have
\begin{equation*}
 p_\sigma(y)
    = \frac{1}{\operatorname{Vol}(\mathcal{M}) (2\pi\sigma^2)^{\frac{D-d}{2}}}
    e^{-\frac{\|v_y\|^2}{2\sigma^2}}
    \frac{1}{\sqrt{\det A_y}}
      \left(1 + O(\sigma\|v_y\| + \sigma^2)\right).
\end{equation*}
Since  $r=O(\sigma)$ and $d(z,\mathcal{M})\le O(\sigma)$, we obtain 
$b=O(\sigma^{-D+d})$ and $s\le O(\sigma^{1-D+d})$. 

By Lemma \ref{lem:nonasymp_hoeffding},
set $\varepsilon=c\,\sigma^{3}$ and choose
$n=C_0\,\sigma^{-2d-5}$.
With this choice, the exponential terms in \eqref{eq:combined_tail} have exponents
\[
-\frac{n b^2}{2M^2}=-\frac{O(\sigma^{-2d-5})O(\sigma^{-2D+2d)})}{O(\sigma^{-2D})}=-O(\sigma^{-5}),
\]
\[
-\frac{n b^2\varepsilon^2}{32 D K^2}
= -\frac{O(\sigma^{-2d-5})O(\sigma^{-2D+2d)})O(\sigma^{6})}{O(\sigma^{2-2D})}
= -O(\sigma^{-1}),
\]
\[
-\frac{n b^4\varepsilon^2}{8\|s\|^2 M^2}
\le -\frac{O(\sigma^{-2d-5})O(\sigma^{-4D+4d)})O(\sigma^{6})}{O(\sigma^{2-2D+2d})O(\sigma^{-2D})}
= -O(\sigma^{-1}).
\]
Hence there exist constants $C_1$, $C_2$ such that 
\[
\bP\big(\|\widehat\mu_n-\widehat\mu_w\| > \varepsilon\big)
\le C_1\exp({-C_2\sigma^{-1}}),
\]
where $\widehat\mu_n = \frac{\sum_{i=1}^n W(y_i)y_i}{\sum_{i=1}^n W(y_i)}=F(z)$ and $\widehat\mu_w=\frac{\mathbb{E}[W(Y)Y]}{\mathbb{E}[W(Y)]}=\mu_z$,
which yields the stated bound
\[
\bP\big(\|F(z)-\mu_z\| \le c\,\sigma^{3}\big)
\ge 1 - C_1\exp({-C_2\sigma^{-1}}),
\]
provided $n\ge C_0\,\sigma^{-2d-5}$. This completes the proof.
\end{proof}

\begin{proof}[Proof of Lemma \ref{Lemma:prob_in_a_ball}]
Recall the local expansion for any $y\in T(\tau-\varepsilon)$ in Theorem~\ref{th:probability noisy}
\begin{equation}\label{eq:local_expansion-lem4.2}
p_\sigma(y)
    = \frac{1}{\operatorname{Vol}(\mathcal{M}) (2\pi\sigma^2)^{\frac{D-d}{2}}}
    e^{-\frac{\|v_y\|^2}{2\sigma^2}}
    \frac{1}{\sqrt{\det A_y}}
      \left(1 + O(\sigma\|v_y\| + \sigma^2)\right).
\end{equation}

For $z\in \Gamma \subset T(\tau-\varepsilon)$ and around its projection $\pi(z)\in\mathcal{M}$, introduce coordinates $(u,n)$ where $u\in\mathbb{R}^d$ are local tangent coordinates on $\mathcal{M}$ and $n\in\mathbb{R}^{D-d}$ represent normal coordinates. 
The corresponding volume element satisfies $dy = J(u,n)\,du\,dn$.
Then
\begin{equation}\label{eq:ball_coord}
\begin{aligned}
\mathbb{P}(Y\in B_D(z,\sqrt{2}r))
&= \int_{B_D(z,\sqrt{2}r)} p_\sigma(y)\,dy\\
&= \int_{\|u-u_z\|\le r'}\!\!\!\int_{\|n-n_z\|\le r''(u)}\!\!\!
p_\sigma(u,n)\,J(u,n)\,dn\,du
\end{aligned}
\end{equation}
where $(u_z,n_z)$ are the coordinates of $z$, and $r',r''(u)$ are quantities comparable to $r$ up to geometric constants.

By~\eqref{eq:local_expansion-lem4.2}, the dominant dependence on $n$ is Gaussian:
\[
(2\pi\sigma^2)^{-(D-d)/2}\,e^{-\|n\|^2/(2\sigma^2)}.
\]
Since $r=O(\sigma)$, the integration in the normal direction over $\|n-n_z\|\le c\sigma$ captures a fixed fraction of the total Gaussian mass.
Hence, 
\begin{align}\label{eq:normal_integral}
c_0\,e^{-\|n_z\|^2/(2\sigma^2)}
\le
\int_{\|n-n_z\|\le c\sigma} (2\pi\sigma^2)^{-(D-d)/2}
 e^{-\|n\|^2/(2\sigma^2)}\,dv.
\end{align}

Substituting~\eqref{eq:normal_integral} into~\eqref{eq:ball_coord}, 
and using that $J(u,n)=1+O(r)$ and 
$\frac{1}{\sqrt{\det A_{\pi(y)}}} = \frac{1}{\sqrt{\det A_{\pi(z)}}} + O(r)$,
we obtain
\[
\mathbb{P}(Y\in B_D(z,\sqrt{2}r))
\ge \frac{c_0e^{-\|n_z\|^2/(2\sigma^2)}}{\text{Vol}(\mathcal{M})\sqrt{\det A_{\pi(z)}}}
   \Big(1+O(r)\Big)\!\!\int_{\|u-u_z\|\le r'} du,
\]
where $c_0$ is the constant from~\eqref{eq:normal_integral}.  
The tangential integral corresponds to the $d$-dimensional Euclidean volume of a ball of radius $\simeq r$, 
that is, $\omega_d r^d$.

Collecting the estimates, we arrive at
\begin{equation}
\mathbb{P}(Y\in B_D(z,\sqrt{2}r))
\ge \frac{\omega_d\,c_0e^{-\|n_z\|^2/(2\sigma^2)}}{\text{Vol}(\mathcal{M})\sqrt{\det A_{\pi(z)}}}
  \,r^d\,(1+O(r)).
\end{equation}
This completes the proof.
\end{proof}

\begin{proof}[Proof of Corollary \ref{Col:Local_sample_size}]
     The number of points $n$ can be viewed as a binomial random variable with size $N$ and probability parameter $p \ge c\sigma^d$. Then by the Chernoff bound (Lemma~\ref{lem:Chernoff bound}),
     $$\bP(n\ge C_1 \sigma^{-5-2d})\geq 1 - 2\exp({ -C_3\sigma^{-5-2d}}).$$
    When $\sigma$ is sufficiently small, the probability will be close to $1$.
\end{proof}

\subsection{Proofs of Section 2}

\begin{proof}[Proof of Theorem \ref{thm:main}]
   By the argument in Theorem \ref{thm:Population-level output manifold}, we may estimate the population mapping $\mu_z$ at each $z\in \Gamma$ as 
\begin{align*}
\mu_z=\frac {\bE(\varphi_r(Y-z)Y)}{\bE(\varphi_r(Y-z))}&=\pi(z)+\frac{d}{2}H_{\pi(z)}\sigma^2+\sigma^2O(\|v_z\|+\sigma\|v_z\|+\sigma^2)\\
&=\pi(z)+\frac{d}{2}H_{\pi(z)}\sigma^2+O(\sigma^3). 
\end{align*}
By Theorem \ref{thm:Concentrate_with_n} and Corollary \ref{Col:Local_sample_size}, for the sample size $N \ge O(\sigma^{-3d-5})$, we prove the equivalent property between $F(z)$ and $\mu_z$, that is,
$$\bP(\|F(z) - \mu_z\|\leq c\sigma^{3}) \geq 1 - C_1\exp\left(-C_2 \sigma^{-1}\right).$$
Hence,     
$$F(z)=\pi(z)+ \frac{d}{2}H_{\pi(z)}\sigma^{2}+O(\sigma^3).$$
This finishes the proof.
\end{proof}

\begin{proof}[Proof of Corollary \ref{cor:uniform estimate}]
    Theorem~\ref{thm:main} ensures that for each fixed $z \in \mathcal Y^o$, the event
    \[E_z:=\{F(z) = \pi(z) + \tfrac{d}{2} H_{\pi(z)} \sigma^2 + O(\sigma^3)\}\]
    holds with probability at least $1 - C_1 \exp({-C_2 \sigma^{-1}})$. Then the joint event
    \[E := \bigcap_{\substack{z \in \mathcal Y^o}} E_z\]
    holds with probability at least
    \[
   P(E) = 1 - P\left(\bigcup_{z \in \mathcal Y^o} E_z^c\right) \ge 1 - \sum_{z \in \mathcal Y^o} P(E_z^c) \ge 1 - n C_1 e^{-C_2 \sigma^{-1}}.
   \]
   Note that
    $$1 - nC_1 e^{-C_2 \sigma^{-1}}\ge 1 - C_1' \exp(-C_2' \sigma^{-c'})$$
    after absorbing the polynomial factor $n\le N=O(\sigma^{-3d-5})$ into the constants $c', C_1', C_2'$.
\end{proof}

\begin{proof}[Proof of Lemma \ref{lem:reach of perturbed manifold}]
Recall Federer's criterion (Lemma \ref{Lemma:ReachCond}) for the reach, without loss of generality, we may assume
\[
0<\tau^{-1} = \sup\left\{ \frac{2 \cdot d(b, T_a\mathcal{M})}{\|a - b\|^2} \;\middle|\; a, b \in \mathcal{M}, a \neq b \right\}.
\]

Let $p, q \in \widehat{\mathcal{M}}$ with $p = x + f(x)$, $q = y + f(y)$ for $x, y \in \mathcal{M}$. Define $\Delta = y - x$ and decompose it as $\Delta = \Delta^\top + \Delta^\perp$, where $\Delta^\top \in T_x \mathcal{M}$ and $\Delta^\perp \in T^{\perp}_x \mathcal{M}$. Note that $\|\Delta^\perp\| = d(y, T_x\mathcal{M})$.

The tangent space of $\widehat{\mathcal{M}}$ at $p$ is given by
\[
T_p \widehat{\mathcal{M}} = \{ x + f(x) + v + Df_x(v) \mid v \in T_x \mathcal{M} \}.
\]
We now estimate the distance from $q$ to $T_p \widehat{\mathcal{M}}$ by
\[
d(q, T_p \widehat{\mathcal{M}}) = \min_{v \in T_x \mathcal{M}} \|\Delta + f(y) - f(x) - v - Df_x(v)\|.
\]
Choosing $v = \Delta^\top$, we obtain
\[
d(q, T_p \widehat{\mathcal{M}}) \le \|\Delta^\perp + f(y) - f(x) - Df_x(\Delta^\top)\|.
\]
Using Taylor expansion for $f\in C^1$, we obtain
\[
f(y) - f(x) = Df_x(\Delta) + O(\|\Delta\|^2) = Df_x(\Delta^\top) + Df_x(\Delta^\perp) + O(\|\Delta\|^2).
\]
Substituting into the distance expression
\[
d(q, T_p \widehat{\mathcal{M}}) \le \|\Delta^\perp + Df_x(\Delta^\perp) + O(\|\Delta\|^2)\|.
\]

Since $f(x) = \frac{d}{2} H_x \sigma^2 + O(\sigma^3)$ and $\mathcal{M}$ is compact, we have
$\|Df_x\| = O(\sigma^2)$ uniformly on $\mathcal{M}$. Then $\|Df_x(\Delta^\perp)\| \leq O(\sigma^2) \cdot \|\Delta^\perp\|$.
   
By Federer's criterion (Lemma \ref{Lemma:ReachCond}) applied to $\mathcal{M}$, we have
$C^{-1}\|\Delta^\perp\|\le \|\Delta\|^2 \leq C\|\Delta^\perp\|$ for some $C>0$.

Combining these estimates, we obtain
\[
d(q, T_p \widehat{\mathcal{M}}) \leq \|\Delta^\perp\|(1 + C_1\sigma^2 + C_2),
\]
for constants $C_1, C_2 > 0$ depending on $\mathcal{M}$.

Now consider the distance between $p$ and $q$ as
\[
\|p - q\| = \|\Delta + f(y) - f(x)\| = \|\Delta\|(1 + O(\sigma^2)),
\]
since $\|f(y) - f(x)\| = O(\sigma^2\|\Delta\|)$.

Applying Federer's criterion to $\widehat{\mathcal{M}}$, we obtain
\[
\frac{\|p - q\|^2}{2 \cdot d(q, T_p \widehat{\mathcal{M}})} \ge \frac{\|\Delta\|^2(1 + O(\sigma^2))}{2\|\Delta^\perp\|(1 + O(1))} = \frac{\|\Delta\|^2}{2\|\Delta^\perp\|} \cdot \frac{1 + O(\sigma^2)}{1 + O(1)}.
\]

Since $\frac{\|\Delta\|^2}{2\|\Delta^\perp\|} \geq \tau$ and $\frac{1 + O(\sigma^2)}{1 + O(1)} \geq 1 - C\sigma^2$ for sufficiently small $\sigma$, we conclude
\[
\frac{\|p - q\|^2}{2 \cdot d(q, T_p \widehat{\mathcal{M}})} \geq \tau(1 - C\sigma^2).
\]

Taking the infimum over all $p \neq q \in \widehat{\mathcal{M}}$ completes the proof.
\end{proof}

\begin{proof}[Proof of Theorem \ref{thm:sample-level output manifold}]
Since $\Gamma$ is a $D$-dimensional tubular neighborhood of $\mathcal{M}$ with radius $C\sigma$, then it admits a finite covering by Euclidean balls
\[
\Gamma \subset \bigcup_{j=1}^M B_D(z_j, \rho),
\]
where $\rho = c_0 \sigma^3$ for some constant $c_0>0$, and the covering number satisfies $M \le C_{\mathrm{cov}} \rho^{-D} = O(\sigma^{-3D})$.

By Corollary~\ref{cor:uniform estimate}, for each center $z_j\in \Gamma$ one has the local expansion
\[
F(z_j) = \pi(z_j) + \tfrac{d}{2} H_{\pi(z_j)} \sigma^2 + O(\sigma^3)
\]
with probability at least $1 - C_1\exp(-C_2\sigma^{-1})$, where $\pi(z_j)$ denotes the projection of $z_j$ onto $\mathcal{M}$ and $H_{\pi(z_j)}$ is the mean curvature vector at $\pi(z_j)$.  
Applying a union bound over all centers and noting that $M\le O(\sigma^{-3D})$ grows only polynomially in $\sigma^{-1}$, we obtain modified constants $C'_1,C'_2,c'>0$ such that
\[
\bP\!\left(
  \forall j \le M, \;
  \|F(z_j) - \pi(z_j) - \tfrac{d}{2}H_{\pi(z_j)}\sigma^2\| \le C\sigma^3
\right)
\ge 1 - C'_1 \exp(-C'_2 \sigma^{-c'}).
\]

Next, we extend this uniform bound from centers $\{z_j\}$ to all $z \in \Gamma$.  
Since $F$ is $C^{\infty}$ on the compact set $\Gamma$, then $F$ is Lipschitz continuous with a constant $L_F $.  
Similarly, both the projection map $\pi$ and the mean curvature field $H_{\pi(\cdot)}$ are $C^1$ on $\Gamma$, implying Lipschitz continuity with constants $L_\pi, L_H$.  
Therefore, for any $z,z'\in\Gamma$ with $\|z-z'\|\le\rho$, we have
\[
\|F(z)-F(z')\| = O(\rho), \quad
\|\pi(z)-\pi(z')\| = O(\rho), \quad
\|H_{\pi(z)}-H_{\pi(z')}\| = O(\rho).
\]
Combining these with the expansion at the centers yields
\[
F(z) = \pi(z) + \tfrac{d}{2} H_{\pi(z)}\sigma^2 + O(\sigma^3),
\quad \text{uniformly for all } z\in\Gamma,
\]
with probability at least $1 - C'_1 \exp(-C'_2 \sigma^{-c'})$. Then the image $F(\Gamma)$ is a smooth $d$-dimensional submanifold of $\mathbb{R}^D$ with high probabilities. 
Since $\|F(z)-\pi(z)\| \le C\sigma^2$ uniformly, we obtain
\[
\sup_{z\in\Gamma} d(F(z),\mathcal{M}) \le C\sigma^2.
\]
In contrast, for every $x\in\mathcal{M}$ there exists $z=x+O(\sigma)v_x\in T_x^\perp \mathcal{M}\subset\Gamma$ such that $F(z)=x+O(\sigma^2)v_x$, which implies that
\[
\sup_{x\in\mathcal{M}} d(x,F(\Gamma)) \le C\sigma^2.
\]
Combining the two bounds establishes
\[
\operatorname{dist}_H(F(\Gamma),\mathcal{M}) \le C\sigma^2
\]
with probability at least $1 - C'_1 \exp(-C'_2 \sigma^{-c'})$.
Moreover, by the argument of Lemma \ref{lem:reach of perturbed manifold}, the reach of $F(\Gamma)$ satisfies 
\[\tau_F\ge \tau(1 - C\sigma^2)>0\]
for some constant $C>0$. 
\newline

\textbf{Acknowledgments.} We are grateful to Junhao Chen for helpful discussions regarding the proofs of Theorems 3.3 and 3.5. Z.Y. has been supported by the	
Singapore Ministry of Education Tier 2 grant A-8001562-00-00 and the Tier 1 grant (A-8000987-00-00 and A-8002931-00-00) at the National University of Singapore; R.L. is a research fellow supported by Tier 2 grant A-8001562-00-00.

\end{proof}

\bibliographystyle{plain}
\bibliography{reference_mf}

\end{document}